\documentclass[reqno,12pt]{amsart}
\usepackage[utf8]{inputenc} 
\usepackage[margin=1in]{geometry}
\usepackage{amsfonts}
\usepackage[all]{xy}
\usepackage{enumitem}
\usepackage{mathtools}
\usepackage{amssymb,amsmath,url}
\usepackage{setspace}
\usepackage{amsthm}
\usepackage{cite}
\usepackage{color}
\usepackage{mathrsfs}
\usepackage{bbm}
\usepackage{tikz}
\usepackage{tikz-cd}

\numberwithin{equation}{section}
\usepackage[normalem]{ulem}
\usepackage{graphicx}
\usepackage{lmodern}
\graphicspath{ {images/} }
\usepackage{pdfpages}
\usepackage{MnSymbol}

\newcommand{\abs}[1]{\left\vert#1\right\vert}
\newcommand{\norm}[1]{\left\Vert#1\right\Vert}
\newcommand{\h}{\mathbb{H}}
\newcommand{\adj}{\mathrm{adj}}
\newcommand{\ol}{\overline}

\newcommand{\wt}{\widetilde}
\newcommand{\rl}{\mathbb{R}}
\newcommand{\cx}{\mathbb{C}}

\newcommand{\D}{\mathbb{D}}
\newcommand{\Z}{\mathbb{Z}}
\newcommand{\Uu}{\mathcal{U}}
\newcommand{\one}{\mathbbm{1}}
\newcommand{\J}{\mathcal{J}}
\newcommand{\ipr}[1]{\left\langle #1 \right\rangle}
\newcommand{\dbar}{\overline{\partial}}

		\newtheorem{thm}{Theorem}[section]

	\newtheorem{lem}[thm]{Lemma}

		\newtheorem{prop}[thm]{Proposition}
	\newtheorem{co}[thm]{Corollary}
	\newtheorem*{thm*}{Theorem}
    	\newtheorem*{lem*}{Lemma}

	\theoremstyle{definition}

\title[Restricted type estimates]{Restricted type estimates on the Bergman projection of some singular domains}
\author{Debraj Chakrabarti}
\address{Department of Mathematics, Central Michigan University, 
Mt. Pleasant, MI 48859,USA}
\email{chakr2d@cmich.edu}
\urladdr{https://people.se.cmich.edu/chakr2d/} 
\author{Zhenghui Huo}
\address{Division of Natural and Applied Sciences, Duke Kunshan University, Kunshan, Jiangsu, 215316, China}
\email{Zhenghui.Huo@dukekunshan.edu.cn}
\urladdr{https://sites.google.com/view/huozhenghui} 
 \thanks{Debraj Chakrabarti was partially supported by a grant from the NSF (DMS--2153907) and a grant from the Simons foundation (No. 706445). Zhenghui Huo was partially supported by NSFC Grant \# 12201265.}       
	\begin{document}
	
	\begin{abstract} We obtain (weighted) restricted type estimates for the Bergman projection operator on monomial polyhedra, a class of domains generalizing the Hartogs triangle. From these estimates,
    we recapture $L^p$ boundedness results of the Bergman projection  on these domains. On some monomial polyhedra, we also discover that the Bergman projection  could fail to be of 
    weak type   $(q_*,q_*)$ where $q_*$ is the right endpoint of
    of the interval of $L^p$-regularity of the domain. 
    
\end{abstract}

 	\maketitle

\section{Introduction}
\subsection{The Bergman projection and the limited range phenomenon} For a domain (open connected subset) $\Omega$ in $\cx^n$ the \emph{Bergman projection} $P_\Omega$ is the orthogonal projection operator from the Hilbert space
$L^2(\Omega)$ of  square 
integrable functions on $\Omega$  to the \emph{Bergman space} $A^2(\Omega)$, the subspace 
of holomorphic  functions in $L^2(\Omega)$. The operator $P_\Omega$ has the integral representation
\[ P_\Omega f(z)= \int_\Omega K_\Omega(z,w)f(w)dV(w) \quad \text{ for } f\in L^2(\Omega),\]
where $dV$ is the Lebesgue measure of $\cx^n$, and
$K_\Omega:\Omega\times \Omega\to \cx$ is the Bergman kernel 
of $\Omega$. Mapping properties of $P_\Omega$  in spaces of smooth functions assume  importance in view of the connection 
with the $\dbar$-Neumann problem as well as the boundary behavior
of holomorphic mappings (see \cite{follandkohn, bellligocka}). The $L^p$-mapping 
properties of $P_\Omega$ have also been an object of considerable
interest (see \cite{ZahJud64,rudin,PhoSte77,McNSte94} etc.), driven partly by an analogy with the classical theory of 
 singular integral operators. On certain classes of smoothly bounded pseudoconvex domains, one can interpret the Bergman projection as a generalized Calderón-Zygmund singular integral operator by introducing appropriate pseudometrics adapted to the boundary geometry (see \cite{jeffbergmansingular}), thereby obtaining boundedness in $L^p$ for $1<p<\infty$ as well as weak type $(1,1)$ estimates (see also \cite{deng}). On the polydisc 
$\D^n$, the  operator $P_{\D^n}$ is of weak type $L(\log^+ L)^{n-1}$ (see \cite{huowickweaktype}), a phenomenon reminiscent of the left endpoint behavior of the tensor product of Hilbert transforms acting on functions on $\rl^n$. 

On many domains, however, the Bergman projection has a very
different behavior. For example, the Bergman projection may have a limited range of $L^p$-regularity, i.e., it is bounded in $L^p$ if and only if $p$ belongs to an open interval $(p_*,q_*)$ containing $2$ where $p_*, q_*$ are Hölder conjugates by the self-adjointness of $P_{\Omega}$. It is also possible that the Bergman projection is bounded in $L^p$ only for $p=2$ (see \cite{EdhMcN16b}). An important early example of limited $p$ range was the nonpseudoconvex but smoothly bounded Barrett worm (see \cite{barrett84}). Another example 
(see \cite{chakzeytuncu, EdhMcN16}) is the \emph{classical Hartogs triangle} 
\begin{equation}
    \label{eq-class-hartogs}
    \h= \{z\in \cx^2: \abs{z_1}<\abs{z_2}<1\},
\end{equation}
a nonsmooth pseudoconvex domain which has long 
served as a source of counterexamples in several complex variables. The Bergman projection $P_\h$ is bounded in $L^p$ if and only if $\frac{4}{3}<p<4$. Subsequently, several authors discovered such ``limited range of $L^p$-regularity" in other domains with geometry
similar to that of $\h$ (see \cite{EdhMcN16,EdhMcN16b,chen17,zhang2,zhang1} etc.).
Most of these examples were shown
in \cite{CJM} to belong to  a class of domains called \emph{monomial polyhedra}. In this paper, we will use these domains as  ``laboratory models" to understand the possible $L^p$ mapping properties of the Bergman projection outside the realm of smoothly bounded domains.  A bounded domain $\Uu_B$  in $\cx^n, n\geq 2$ is called a
{monomial polyhedron} if it can be expressed as
\begin{equation}
\label{eq-udef}
\Uu_B=\left\{z\in\cx^n: \text{ for } 1\leq j \leq n,\quad  \prod_{k=1}^n\abs{z_k}^{b^j_k} <1  \right\},
\end{equation}
where $B\in \Z^{n\times n}$ is an $n\times n$ matrix of integers, and the entry 
of $B$ at the $j$-th row and $k$-th column is denoted by $b^j_k$.
If $B=\begin{pmatrix} 1&-1\\0&\phantom{-}1\end{pmatrix}$,
we see that $\Uu_B=\h$. We will say that $\Uu_B$ is nontrivial if it is not the polydisc $\D^n$ (which corresponds to taking $B$ to be the identity matrix).
It was shown in \cite{CJM} that on a nontrivial monomial polyhedron $\Uu_B$ the Bergman projection $P_{\Uu_B}$  has limited range of $L^p$-boundedness. Monomial polyhedra, and \emph{quotient domains} that generalize them, have been studied by several authors 
(\cite{CKY19,dallara, MR4627650,MR4630461,summerpaper2023,huowicksymmetrized,qin} etc.) from various points of view.



\subsection{Weighted restricted type estimates} In this paper we  continue the study of the $L^{p}$ mapping behavior
of the Bergman projection on the monomial polyhedron $\Uu_B$.  
As noted above, there are Hölder-conjugate numbers
$p_*, q_*$ with $p_*<2<q_*$ such that  $P_{\Uu_B}$ is of strong  type 
$(p,p)$ (i.e. bounded in $L^p$) if and only if $p_*<p<q_*$. 
Several authors have studied the behavior of $P_{\Uu_B}$ at $q_*$ for specific $B$. For the Hartogs
triangle, and certain of its generalizations, 
the operator $P_{\Uu_B}$ is of weak type $(q_*,q_*)$ (see \cite{huowickweaktype,weak3, weak2, weak1,liwang, weak4}), i.e.,
$P_{\Uu_B}$ is bounded from the space $L^{q_*}(\Uu_B)$ to the Lorentz space $L^{q_*,\infty}(\Uu_B)$. Classical results of Lorentz space theory (recalled in Proposition~\ref{prop-duality} below) then suggest that the Bergman projection is of so-called \emph{restricted type} $(p_*, p_*)$, 
a notion that we now recall (see \cite{steinweissbook}). 

Let $(M,\mu)$ and $(N,\nu)$ be measure spaces (we suppress the $\sigma$-algebras from the notation). For
$p\geq 1$, a mapping $T$ is of \emph{restricted type $(p,p)$}  if $T$ is a linear map from 
the space of  simple functions on $M$ to the space  of measurable functions on $N$, and there is a constant $C>0$ such that whenever 
$E\subset M$ is measurable  with $\mu(E)<\infty$, we have 
\begin{equation}
    \label{eq-restricted}\norm{T(1_E)}_{L^p(N,\nu)} \leq C \norm{1_E}_{L^p(M,\mu)}.
\end{equation}
Here we use $1_E$ to denote the indicator (or characteristic) function of $E$, the function that is 1 on $E$ and 0 on the complement of $E$. One can also think of operators of restricted
type $(p,p)$ as those which have continuous extensions to operators from the Lorentz space $L^{p,1}(M,\mu)$
to the space $L^p(N,\nu)$ (see \cite{hunt}). 

Similarly, a linear operator $T$ mapping simple functions on $M$ to measurable functions on $N$ is called of \emph{restricted weak type $(p,p)$} if there is a $C>0$ such that  for each measurable subset $E\subset M$ of finite measure,  and for each $y>0$:
\begin{equation}
    \label{eq-def-rwt} \nu\{x\in N: \abs{T(1_E)(x)}> y\} \leq  \left(\frac{C}{y} \norm{1_E}_{L^p(M,\mu)} \right)^p= \frac{C^p}{y^p} {\mu(E)}.
\end{equation}
From the Chebyshev-Markov inequality one can see that if $T$ is of restricted type $(p,p)$, then it is of restricted weak type $(p,p)$ as well. 

In view of duality, it is natural to look at 
restricted type estimates on the Bergman projection with $p$ value
at the left-endpoint of the interval of $L^p$-regularity of $P_{\Uu_B}$. In fact, it makes sense to consider
the \emph{positive Bergman operator} $P_{\Uu_B}^+$, which dominates the Bergman projection and has similar mapping properties,  where for domain $\Omega$,   $P_\Omega^+$ is 
the operator whose action on a measurable function $f$ is given by
\[P_\Omega^+ f(z)=\int_\Omega \abs{K_\Omega(z,w)}f(w)dV(w),  \]
provided the integral converges. The operator $P_\Omega^+$ has 
the advantage that its adjoint in weighted $L^p$-spaces is easy to compute (see \eqref{eq-padjoint} below).
We  show that  restricted type estimates hold at $p_*$ provided $\Uu_B$
satisfies a further condition, that 
a certain invariant $m$ (defined in  \eqref{eq-m}) associated with 
the domain $\Uu_B$ is equal to 1 (see Corollary~\ref{cor-m=1} below).  
 When $m\not=1$,
we still obtain restricted type estimates, but now weighted ones, 
between \emph{weighted measure spaces} on  $\Uu_B$ (see Theorem~\ref{thm-genest} below). For a nonnegative measurable function $v$ on a domain $\Uu\subset \cx^n$, by the weighted measure space $(\Uu, v)$, we mean the measure space with the $\sigma$-algebra of (Lebesgue)-measurable subsets of $\Uu$, and
with the measure of a measurable $E\subset \Uu$ given by $ v(E)= \int_E vdV.$


\subsection{Main results} We will always assume that the dimension $n\geq 2$. Let 
$\Delta=\adj B=(\det B)B^{-1}$ be the adjugate matrix of the integer matrix $B\in \Z^{n\times n}$ that occurs in the definition \eqref{eq-udef} of the domain 
$\Uu_B$, where $\Z^{a\times b}$ denotes the set of integer matrices of $a$ rows and $b$ columns. By Cramer's rule, $\Delta$ has integer entries, and from the fact that $\Uu_B$ is a 
bounded open set, it easily follows that each entry of $\Delta$ is nonnegative and $\det \Delta\not=0$ (see \cite[Proposition~3.2]{CJM}).
 Let  For $1\leq j \leq n$, denote  by $\delta_j\in \Z^{n\times 1}$ the $j$-th column of the matrix $\Delta$. Let 
\[ a_j= \frac{1}{\gcd \delta_j} \delta_j\in \Z^{n\times 1},\]
where the positive integer $\gcd \delta_j$  is the greatest common 
divisor of the $n$ entries of the column matrix $\delta_j$ (which is nonzero since $\det \Delta\not=0$).
We set
\begin{equation}
    \label{eq-adef} A= \begin{pmatrix}
        \vert & \cdots & \vert\\
        a_1 & \cdots & a_n\\
        \vert & \cdots & \vert
    \end{pmatrix}\in \Z^{n\times n}
\end{equation}
to be the $n\times n$ matrix of nonnegative integers whose $j$-th column is $a_j$, obtained by dividing the column vector $c_j$ by the gcd of its elements. 
The geometric significance of the matrix $A$ is that it is the symbol of a monomial branched covering map $z\mapsto z^A$ of $\Uu_B$ by a punctured polydisc, see Proposition~\ref{prop-mono2} below.
Let  \begin{equation}
    \label{eq-onedef}
    \one = (1,\dots, 1)\in \Z^{1\times n}
\end{equation}
be the $1\times n$ row matrix whose entries are all 1's.
Next, we introduce several quantities that are crucial in our results:
\begin{enumerate}[wide]
\item Let $m$ be 
the \emph{number of largest elements} in the row-vector of integers
\[\one A=(\one a_1, \one a_2, \dots, \one a_n)\in \Z^{1\times n},\] i.e.,
denoting by $\#S$ the number of elements of a set $S$:
\begin{equation}
    \label{eq-m}
     m= \#\{j:\one a_j = \max_k\{ \one a_k, 1\leq k \leq n\}\}.
\end{equation}
\item Let \begin{equation}
    \label{eq-pstar1}
     p_*= \max_j \frac{2\one a_j}{\one a_j+1}.
\end{equation}
It was shown in \cite{CJM} that $p_*$ is the left endpoint of the interval of $L^p$-boundedness of the Bergman projection on $\Uu_B$. If $\Uu_B$ is nontrivial, then $1<p_*<2.$
\item  For a tuple $(\alpha_1,\dots, \alpha_n)\in \rl^n$ of  real numbers 
let  $\rho_{\alpha}:\cx^n\to \rl$ be  the function
\begin{equation}
    \label{eq-rhodef}\rho_\alpha(z)= \prod_{j=1}^n \abs{z_j}^{\alpha_j} =\abs{z_1}^{\alpha_1}\cdot \dots\cdot \abs{z_n}^{\alpha_n}.
\end{equation}
If $\alpha\in \Z^n$, notice that $\rho_\alpha(z)=\abs{z^\alpha}$, where, using standard multi-index notation $z^\alpha=\prod_{j=1}^nz_j^{\alpha_j}$.
Define a  function $w$ on $\Uu_B$ by setting
\begin{equation}
    \label{eq-vdef-poly}
    w=\rho_{\one-\one A^{-1}}.
\end{equation}
It is not difficult to see that $0\leq w \leq 1$ on 
$\Uu_B.$

\end{enumerate}
Our first result is the following weighted restricted type estimate:
\begin{thm}
    \label{thm-genest} Let $\Uu_B$ be a nontrivial monomial  polyhedron. 
     Then positive Bergman operator 
      $P^+_{\Uu_B}$ is of restricted type
      $(p_*,p_*)$ from  
      the weighted measure space $(\Uu_B,(-\log w)^{(m-1)(p_*-1)})$ to   the weighted measure space $(\Uu_B,w^{p_*-2})$, i.e., there is a constant $C>0$ (depending only on the domain $\Uu_B$) such that for each measurable $E\subseteq \Uu_{B}$ 
    \begin{equation}
        \label{eq-genest}
        \norm{P^+_{\Uu_{B}}(1_E)}_{L^{p_*}(\Uu_B,w^{p_*-2})}\leq C\norm{1_E}_{L^{p_*}(\Uu_B, (-\log w)^{(m-1)(p_*-1)}) }.
    \end{equation}
\end{thm}

As mentioned earlier, 
a duality argument  may be
applied to (\ref{eq-genest}), yielding an  inequality at the right endpoint:
\begin{prop}\label{prop-sawyer} Let $\mu$ be the measure on $\Uu_B$ given by \[\mu(E)=\int_E (-\log w)^{(m-1)(p_*-1)}dV, \quad E\subset \Uu_B \text{ measurable}\] where $w$ is as in \eqref{eq-vdef-poly}. There is a constant $C>0$ such that for each measurable function $f$ on $\Uu_B$ and for each $y>0$ we have
\begin{equation} \label{eq-sawyer}   \mu\left\{z\in\Uu_B: {\abs{P^+_{\Uu_B}f(z)}}>{(-\log w(z))^{(m-1)(p_*-1)}}\cdot y \right\}\leq \frac{C}{y^{q_*}}\int_{\Uu_B}\abs{f}^{q_*} w^{\frac{2-p_*}{p_*-1}}dV,
\end{equation}
where $q_*$ is the Hölder conjugate of $p_*$.
\end{prop}

Inequality \eqref{eq-sawyer} indicates possible distinct mapping behaviors of the Bergman projection
between the cases $m=1$ and $m>1$.  When $m=1$,  there is a unique largest element in the $n$-tuple
$ (\one a_1,\dots, \one a_n)$. 
In this case, the  estimate \eqref{eq-sawyer} can be reduced to be an unweighted one and can then
be used in an interpolation argument to recover $L^p$ regularity results. Thus Theorem~\ref{thm-genest} leads to the following corollary via Proposition~\ref{prop-sawyer}:
\begin{co}\label{cor-m=1}
    Suppose that the nontrivial monomial polyhedron $\Uu_B$ satisfies the 
    condition $m=1$,  let $p_*$ be as in \eqref{eq-pstar1}, and let $q_*$ be the Hölder conjugate of $p_*$. Then $P_{\Uu_B}^+$ and $P_{\Uu_B}$ are
    \begin{enumerate}
        \item of restricted type $(p_*,p_*)$.
        \item of strong type $(p,p)$ for each $p_*<p<q_*$.
        \item of weak type $(q_*,q_*).$
    \end{enumerate}
\end{co}
Corollary \ref{cor-m=1} recovers previous known $L^p$ boundedness results  on many interesting domains including the 
classical Hartogs triangle $\h$ of \eqref{eq-class-hartogs}, the
generalized Hartogs triangles $\{\abs{z_1}^{\frac{a}{b}}<\abs{z_2}<1\}\subset\cx^2$, where $a,b$ are coprime positive integers, and more generally the $n$-dimensional Hartogs triangles considered in \cite{park,chen17,zhang1}:
\[ \{z\in \cx^n: \abs{z_1}^{k_1}< \abs{z_2}^{k_2}< \dots < \abs{z_n}^{k_n} <1\},\]
with $(k_1,\dots, k_n)$ being a tuple of positive integers of gcd 1. 

When $m>1$, on the other hand, the weight involving $-\log w$ in (\ref{eq-genest}) cannot be dropped, thus suggesting the possible failure of both unweighted restricted type $(p_*,p_*)$ and weak type $(q_*,q_*)$ estimates for $P_{\Uu_B}$. We will prove the following:
\begin{thm}
    \label{thm-irreg1}
There are monomial polyhedra $\Uu_B$  with the Bergman projection $P_{\Uu_B}$ not of restricted  weak type $(p_*,p_*)$,
where $p_*$, as in \eqref{eq-pstar}, is the left endpoint of the interval of $L^p$-boundedness. 
\end{thm}
An example of  such a $\Uu_B$ is
\begin{equation}
    \label{eq-Bpoly}
    \Uu_B=\left\{z\in \cx^3: \abs{z_1z_3}<\abs{z_2}<\abs{z_1}<1\right\},
\end{equation}
for which $p_*=\frac{4}{3}$ 
(see Proposition~\ref{prop-example}.)
Therefore, for the domain $\Uu_B$ in \eqref{eq-Bpoly}, neither the Bergman projection $P_{\Uu_B}$ nor
the positive Bergman projection $P_{\Uu_B}^+$ is  of restricted weak type $(p_*,p_*)$.  It follows that both $P_{\Uu_B}$ and $P_{\Uu_B}^+$ are not of restricted type $(p_*, p_*)$, nor of weak type $(p_*, p_*)$ as well.
Since $P_{\Uu_B}$ is not of restricted type $(p_*,p_*)$,  a duality argument (see Proposition~\ref{prop-duality}) then implies that 
$P_{\Uu_B}$ is not of weak type $(q_*,q_*)$ . This result disagrees with  the claim in a recent paper of Tang, Tu and Zhang (see \cite[Theorem~1.2]{weak1}) that on each monomial polyhedron, the Bergman projection is of weak type $(q_*, q_*)$, where $q_*$ is the right endpoint of the interval of $L^p$-boundedness. 
The error there stems from the last inequality on page 14 of their paper
 \[ 2(\one a_j-1)+2\one a_{j_0} \frac{(\sigma-2\one B_{-})a_j}{(2\one B_{-}-\sigma)a_{j_0}}>-2, \quad j=1,\dots, n, j\not= j_0\]
which does not hold if $m>1$, because then the left hand side is equal to $-2$. 
When $m>1$, extra $\log$ terms will appear in the relevant integral computations, eventually making the conclusion false. 
 However, the argument in \cite{weak1} is correct for the case $m=1$ and therefore gives a direct proof 
 of part  (3) of Corollary~\ref{cor-m=1} above, without recourse to duality arguments. In a future publication, we plan to analyze in detail the weak type $(p_*,p_*)$ irregular behavior of the Bergman projection for general monomial polyhedra with $m>1$. This question is especially interesting in dimension $n=2$, where in all known cases, the Bergman projection is of weak type $(q_*, q_*).$

It is worth noting that for the case $m>1$  we  can no longer adopt our weighted estimates at the endpoint $p=p_*$ in an interpolation argument to recover unweighted $L^p$ strong type estimates of $P_{\Uu_B}$ (as 
in case (2) of Corollary~\ref{cor-m=1} above). Nevertheless, the \emph{method} in the  proof of Theorem~\ref{thm-genest}  gives a unified approach to these results regardless of the value of $m$. 

\begin{prop} \label{prop-restricted} For each nontrivial monomial
polyhedron $\Uu_B$, the operator
 $P^+_{\Uu_{B}}$ is of restricted type $(p,p)$ for $p\in (p_*,2]$.
\end{prop}
As a consequence, we recapture a strengthened version of the main result of \cite{CJM}:
\begin{co}
    \label{co-cjm}
    For $p\in (p_*,q_*)$, both $P^+_{\Uu_{B}}$ and
    $P_{\Uu_{B}}$ are of strong type $(p,p).$
\end{co}

 Our paper is organized as follows: In Section~\ref{sec-duality}, we recall results regarding  duals of Lorentz spaces and prove Proposition~\ref{prop-sawyer}, Corollaries~\ref{cor-m=1} and \ref{cor-m=1} assuming Theorem~\ref{thm-genest} and Proposition~\ref{prop-restricted}. In Section~\ref{sec-polydisc}, we establish weighted restricted type estimates for the Bergman projection on punctured polydiscs. In Section~\ref{sec-pullback}
 we pull back   restricted estimates for $P_{\Uu_B}$
 via proper monomial mappings  into weighted estimates of the Bergman projection on punctured polydiscs so that results in Section~\ref{sec-polydisc} apply and Theorem~\ref{thm-genest} and Proposition~\ref{prop-restricted} are proved. In Section~\ref{sec-irreg}, we give  a family of monomial polyhedra in Theorem \ref{thm-irreg2} where the Bergman projection fails to be of restricted weak type $(p_*,p_*)$, hence proving Theorem \ref{thm-irreg1}.
\subsection{Acknowledgments} We  thank
 Sagun Chanillo, Luke Edholm, Brett Wick, Liding Yao and Shuo Zhang for enlightening conversations that helped 
 improve this paper. 

\section{Application of duality and interpolation} 
\label{sec-duality}
\subsection{Some results of functional analysis}  Let $(X,\mu)$ be a measure space, where  $\mu$ is   $\sigma$-finite,  and let $1< p <\infty$.  The space $L^{p,\infty}(X,\mu)$ of
weak-$L^p$ functions can be defined by the finiteness 
condition $\norm{f}^*_{p,\infty}<\infty,$
where for a measurable function $f$ on $X$, the extended real number $\norm{f}^*_{p,\infty}$ is defined by
\begin{equation}
    \label{eq-quasi}
    \norm{f}^*_{p,\infty}=\sup_{y>0} \left(y\cdot \mu\left\{x\in X: \abs{f(x)}>y\right\}^{\frac{1}{p}}\right).
\end{equation}
Then $\norm{\cdot}_{p,\infty}^*$ satisfies all the defining properties of a norm expect the triangle inequality. Instead, for measurable $f$ and $g$, we only have
$$\norm{f+g}_{p,\infty}^* \leq 2^{\frac{1}{p}}\left(\norm{f}_{p,\infty}^*+ \norm{g}_{p,\infty}^*\right).$$ The ``quasinorm" $\norm{\cdot}_{p,\infty}^*$ defines a linear topology on the vector space $L^{p,\infty}(X,\mu)$. 
It is well-known that 
this topology is in fact normable. Two equivalent norms on $L^{p,\infty}(X,\mu)$ are
\begin{align}
     \label{eq-equivalent-norm-1}
    &\norm{f}_{p,\infty}=\sup\left\{\mu(E)^{-\frac{1}{p'}} \abs{\int_E f d\mu}: E \text{{ measurable, }} \mu(E)<\infty\right\},\\
      \label{eq-equivalent-norm-2}
    &\norm{f}'_{p,\infty}=\sup\left\{\mu(E)^{-\frac{1}{p'}} \int_E\abs{f} d\mu: E \text{ \emph{measurable}, } \mu(E)<\infty\right\}.
\end{align}
See for example \cite[Exercise 1.1.12]{grafakos1} for the equivalence of $\norm{\cdot}'_{p,\infty}$ and $\norm{\cdot}^*_{p,\infty}$. As to the equivalence of $\norm{\cdot}'_{p,\infty}$ and $\norm{\cdot}_{p,\infty}$, notice that for each measurable function $f$, we have $\norm{f}_{p,\infty}\leq\norm{f}'_{p,\infty}$ from their definitions.  When
$f$ is nonnegative or nonpositive, the reverse inequality holds, and we see that 
$ \norm{f}'_{p,\infty}\leq 4\norm{f}_{p,\infty}$ by splitting the real and
imaginary parts of $f$ into positive and negative parts.

We now consider two $\sigma$-finite measure spaces
$(M,\mu)$ and $(N,\nu)$ and let $T$ be a linear operator
mapping a vector space of measurable functions on $(M,\mu)$ into measurable functions on $(N,\nu).$
We can define the  ``weak type $(p,p)$ operator norm"
by setting
\[\norm{T}_{\mathrm{WT}(p)}=\sup\left\{ \frac{\norm{Tf}_{L^{p,\infty}(N,\nu)}}{\norm{f}_{L^p(M,\mu)}}: f\in L^p(M,\mu), f\not=0\right\},\]
which is just the operator norm of $T$ as a linear mapping from the Banach space $L^p(M,\mu)$ into the Banach space
$\left(L^p(N,\nu), \norm{\cdot}_{L^{p,\infty}(N,\nu)}\right).$ 
The operator 
$T$ is of weak type $(p,p)$ if and only
if $\norm{T}_{\mathrm{WT}(p)}<\infty$.
Similarly, we can define the ``restricted type $(q,q)$ operator norm" of $T$
by setting 
\[ \norm{T}_{\mathrm{RT}(q)} = \sup \left\{\frac{\norm{T1_E}_{L^q(N,\nu)}}{\norm{1_E}_{L^q(M,\mu)}}: E \subset X \text{ measurable}, 0<\mu(E)<\infty\right\}.\]
Again $T$ is of restricted type $(q,q)$ if and  only if  $\norm{T}_{\mathrm{RT}(q)}<\infty,$ and it is possible to interpret this as an operator norm between Lorentz spaces.

We will also need the notion of the \emph{adjoint} of the operator $T$. This is a linear operator $T^*$ mapping a vector space of measurable functions on 
$(N,\nu)$ into the space of measurable functions on $(M,\mu)$, and satisfies
$ \displaystyle{\int_{N}(Tf)\ol{g}d\nu = \int_{M} f\cdot \ol{T^* g}d\mu.}$
 In case $T$ is
of restricted type $(p,p)$, for $f$ simple on $(M,\mu)$ and $g\in L^{p'}(N,\nu),$ the existence 
of $T^*$ follows from a standard argument
involving the Radon-Nikodym theorem (see 
\cite[p. 265]{steinweiss1}). For $T$ of weak type $(p,p)$, for $f\in L^{p'}(M,\mu)$ and $g$ simple on $(N,\mu)$, 
the existence of $T^*$ also can be proved
in a similar way. Notice that we have $(T^*)^*=T$. 
We will now recall the 
following duality statement between the 
weak type norm and the restricted type norm:

\begin{prop}\label{prop-duality}
let $T$ be a linear operator
mapping simple functions on $(M,\mu)$ into measurable functions on $(N,\nu)$,
let $1<p<\infty$ and let $p'$ be the Hölder conjugate of $p$. Then
 \begin{equation}
        \label{eq-operatornorms}
        \norm{T}_{\mathrm{RT}(p')}= \norm{T^*}_{\mathrm{WT}(p)}.
    \end{equation}

   
\end{prop}
\begin{proof} 
In \eqref{eq-operatornorms}, it is allowed that the common value of these norms is infinite. 

For a measurable $E\subset M$ of finite nonzero measure we have by duality of the space $L^p(N,\nu)$ with the space $L^{p'}(N,\nu)$
    \begin{align*}
        \norm{T1_E}_{L^{p'}(N,\nu)}&= \sup\left\{\frac{ \abs{\int_N (T1_E)\ol{f}d\nu}}{\norm{f}_{L^p(N,\nu)}}: f\in L^p(N,\nu), f\not =0 \right\}\\
        &= \sup\left\{\frac{ \abs{\int_M 1_E(\ol{T^*f})d\mu}}{\norm{f}_{L^p(N,\nu)}}: f\in L^p(N,\nu), f\not =0 \right\}\\
        &= \sup\left\{\frac{ \abs{\int_E {T^*f} d\mu}}{\norm{f}_{L^p(N,\nu)}}: f\in L^p(N,\nu), f\not =0 \right\}\\
        &=\sup\left\{\frac{1}{\norm{f}_{L^p(N,\nu)}}\frac{ \abs{\int_E {T^*f} d\mu}}{\norm{1_E}_{L^{p'}(M,\mu)}}: f\in L^p(N,\nu), f\not =0 \right\}\cdot \norm{1_E}_{L^{p'}(M,\mu)}\\
        &=\sup\left\{\frac{1}{\norm{f}_{L^p(N,\nu)}} \mu(E)^{-\frac{1}{p'}}\abs{\int_E {T^*f} d\mu}: f\in L^p(N,\nu), f\not =0 \right\}\cdot \norm{1_E}_{L^{p'}(M,\mu)}\\
        &\leq \sup\left\{\frac{\norm{T^*f}_{L^{p,\infty}(M,\mu)}}{\norm{f}_{L^p(N,\nu)}}: f\in L^p(N,\nu), f\not =0 \right\}\cdot \norm{1_E}_{L^{p'}(M,\mu)}\\&= \norm{T}_{\mathrm{WT(p)}}\cdot \norm{1_E}_{L^{p'}(M,\mu)}.
    \end{align*}
    It follows that we have
    $ \norm{T}_{\mathrm{RT}(p')}\leq\norm{T^*}_{\mathrm{WT}(p)}.$
    For the reverse inequality one can proceed similarly
starting from the definition of $  \norm{T^*f}_{L^{p,\infty}(M,\mu)}$, and change the order of taking suprema 
to yield $ \norm{T^*f}_{L^{p,\infty}(M,\mu)} \leq \norm{T}_{\mathrm{RT}(p')}\cdot \norm{f}_{L^p(N,\nu)}$. We leave the details to the reader.
\end{proof}
\subsection{Consequences of Theorem~\ref{thm-genest} and Proposition~\ref{prop-restricted}}

\begin{proof}[Proof of Proposition~\ref{prop-sawyer}]
    Let us set $\mu_0= (-\log w)^{(m-1)(p_*-1)}$ and $\nu_0=w^{p_*-2}$, so that the measures $\mu$ and $\nu$ on $\Uu_B$ can be set to be $d\mu=\mu_0dV$ and $d\nu=\nu_0dV.$
Then thanks to Theorem~\ref{thm-genest},
the operator $P_{\Uu_B}^+$ is of restricted type $(p_*,p_*)$ as an operator from functions on $(\Uu_B, \mu)$
to functions on $(\Uu_B,\nu)$, so that $\norm{P^+_{\Uu_B}}_{\mathrm{RT}(p_*)}<\infty$. Therefore, by Proposition~\ref{prop-duality}, we have 
$\norm{(P^+_{\Uu_B})^*}_{\mathrm{WT}(q_*)}<\infty$ where $(P^+_{\Uu_B})^*$ is the adjoint of $P_{\Uu_B}^+$ mapping functions on $(\Uu_B,\nu)$ to functions on $(\Uu_B, \mu)$  with
$\displaystyle{ \int_{\Uu_B}(P^+_{\Uu_B}f) \ol{g}d\nu = \int_{\Uu_B}f\cdot\ol{(P_{\Uu_B}^+)^* g}\,d\mu}$. Note that
\begin{align}
    \int_{\Uu_B}P^+_{\Uu_B}f(z) \ol{g(z)}d\nu(z)&=\int_{\Uu_B}\left( 
    \int_{\Uu_B} \abs{K_{\Uu_B}(z,\zeta)}f(\zeta)dV(\zeta)\right)\ol{g(z)}\nu_0(z)dV(z)\nonumber\\
    &= \int_{\Uu_B}f(\zeta)\ol{\left(\frac{1}{\mu_0(\zeta)}\int_{\Uu_B}\abs{K_{\Uu_B}(z,\zeta)}\nu_0(z)g(z)dV(z) \right)} d\mu(\zeta),\label{eq-padjoint}
\end{align}
where the change in order of integration can be justified (when $f,g$ are nonnegative) by Tonelli's theorem, and the general justification follows by linearity. Thanks to the symmetry of the kernel $\abs{K_{\Uu_B}}$ defining $P_{\Uu_B}^+$ we have the representation
\[ (P_{\Uu_B}^+)^* g= \frac{1}{\mu_0}P_{\Uu_B}^+(\nu_0 g),\]
for measurable $g$, and we also have the weak type estimate 
\[ \norm{\frac{1}{\mu_0}P_{\Uu_B}^+(\nu_0 g)}^*_{L^{q_*,\infty}(\Uu_B,\mu)} \leq C \norm{g}_{L^{q_*}(\Uu_B,\nu)},\]
where $C$ is a constant independent of $g$ and the left hand side uses the quasinorm \eqref{eq-quasi}. Setting $f=\nu_0g,$ we see that 
\[\norm{\frac{1}{\mu_0}P_{\Uu_B}^+(f)}^*_{L^{q_*,\infty}(\Uu_B,\mu)} \leq C \norm{\frac{f}{\nu_0}}_{L^{q_*}(\Uu_B,\nu)}, \]
which clearly holds for all measurable $f$. Now
\begin{align*}
    \norm{\frac{f}{\nu_0}}_{L^{q_*}(\Uu_B,\nu)}^{q_*} = \int_{\Uu_B} \abs{\frac{f}{\nu_0}}^{q_*}\nu_0 dV = \int_{\Uu_B} \abs{f}^{q_*}\nu_0^{1-q_*} dV=\int_{\Uu_B}  \abs{f}^{q_*}\nu_0^{\frac{1}{1-p_*}} dV = \int_{\Uu_B}  \abs{f}^{q_*}w^{\frac{p_*-2}{1-p_*}} dV.
\end{align*}
The result follows from the definition of the quasinorm \eqref{eq-quasi}. 
\end{proof}


\begin{proof}[Proof of Corollary~\ref{cor-m=1}]
It is sufficient to prove the statements for $P_{\Uu_B}^+$, since for each $p$, $\norm{P_{\Uu_B}(1_E)}_{L^p(\Uu_B)} \leq \norm{P_{\Uu_B}^+ (1_E)}_{L^p(\Uu_B)}$ for each measurable subset $E\subset \Uu_B.$ Since $m=1$, the weight in Theorem~\ref{thm-genest} on the source is reduced to 
\begin{equation}
    \label{eq-weq1}
    (-\log w)^{(1-1)(p_*-1)} \equiv 1,
\end{equation}
so for $E\subset \Uu_B$ is measurable:
\begin{align*}
\norm{P_{\Uu_B}^+ (1_E)}_{L^{p_*}(\Uu_B)}&\leq  \norm{P_{\Uu_B}^+ (1_E)}_{L^{p_*}(\Uu_B,w^{p_*-2})}
&\text{since $w^{p_*-2}\geq 1$}\\
&\leq C \norm{1_E}_{L^{p_*}(\Uu_B)}. & \text{by \eqref{eq-genest}}
\end{align*}
This establishes conclusion (1). 

For conclusion (3), notice that when $m=1$, thanks to \eqref{eq-weq1},  inequality \eqref{eq-sawyer} shows that for each $y>0$
\begin{align*}
 \abs{\left\{z\in\Uu_B:\abs{P^+_{\Uu_B}f(z)}>  y \right\}}&=
   \mu\left\{z\in\Uu_B:\abs{P^+_{\Uu_B}f(z)}>  y \right\}\\&\leq \frac{C}{y^{q_*}}\int_{\Uu_B}\abs{f}^{q_*} w^{\frac{2-p_*}{p_*-1}}dV\\
   & \leq \frac{C}{y^{q_*}}\int_{\Uu_B}\abs{f}^{q_*}\,dV,
\end{align*}
where the last step follows since $1<p_*<2,$ and $0\leq w <1$. Thus, $P^+_{\Uu_B}$ is of weak type $(q_*, q_*).$

For conclusion (2), note that $P^+_{\Uu_B}$ is of restricted type $(p_*,p_*)$ and of weak type $(q_*,q_*)$.
Therefore it is of both restricted weak type $(p_*,p_*)$
and restricted weak type $(q_*,q_*)$. Then by the restricted weak-type interpolation theorem (see \cite{steinweiss1,steinweissbook}), $P_{\Uu_B}$ is of strong type $(p,p)$ whenever $p_*<p<q_*$. 
\end{proof}

\begin{proof}[Proof of Corollary~\ref{co-cjm}]
  Again,  it suffices to prove this for $P_{\Uu_B}^+$.
  For the given $p\in (p_*, q_*)$ choose a $p_0$ in the interval $(p_*,2)$ such that $p\in (p_0, p_0')$ where
  $p_0'$ is the Hölder conjugate of $p_0$. By Proposition~\ref{prop-restricted}, the operator $P_{\Uu_B}^+$ is of restricted type $(p_0,p_0)$, i.e.,
  $\norm{P_{\Uu_B}^+}_{\mathrm{RT}(p_0)}<\infty$, where we think of $P_{\Uu_B}^+$ as a linear operator on functions on the measure space $(\Uu_B,\abs{\cdot})$ to itself, were $\abs{\cdot}$ is Lebesgue measure.
  Therefore, by Propostion~\ref{prop-duality}, the adjoint operator $(P_{\Uu_B}^+)^*$ is of weak type $(p_0', p_0')$. In this case
  $(P_{\Uu_B}^+)^*=P_{\Uu_B}^+$, since as in \eqref{eq-padjoint}
  \begin{align*}
    \int_{\Uu_B}P^+_{\Uu_B}f(z) \ol{g(z)}dV(z)&=\int_{\Uu_B}\left( 
    \int_{\Uu_B} \abs{K_{\Uu_B}(z,\zeta)}f(\zeta)dV(\zeta)\right)\ol{g(z)}dV(z)\\
    &= \int_{\Uu_B}f(\zeta)\ol{\left(\int_{\Uu_B}\abs{K_{\Uu_B}(z,\zeta)}g(z)dV(z) \right)} d V(\zeta),
\end{align*}
  and the absolute Bergman kernel is symmetric. 
  Therefore $P_{\Uu_B}^+$ is of restricted weak type $(p_0,p_0)$ and 
  restricted weak type $(p_0', p_0')$ and therefore it is of 
  strong type $(p,p)$, again by restricted weak-type interpolation theorem (see \cite{steinweiss1,steinweissbook}). 
\end{proof}

\section{An estimate on the polydisc}\label{sec-polydisc}
The proof of Theorem~\ref{thm-genest}
uses the fact that $\Uu_B$ is a so-called
\emph{quotient domain}, i.e., there is a ramified covering map $\phi:\Omega\to \Uu_B$ where $\Omega$ is a so-called 
punctured polydisc (i.e. a product of 
unit discs and punctured unit discs). This allows us to ``pull back" the problem to the covering domain $\Omega.$ Since an integral operator on $\Omega$ can be thought of as an integral operator on the polydisc $\D^n$, we can reduce the problem to an estimate on the polydisc. In Section~\ref{sec-polydisc} below, we establish
this equivalent estimate. Using it, we will complete the proof for Theorem~\ref{thm-genest} in Section~\ref{sec-pullback}. 
\subsection{Notation and statement of result}
Let $\displaystyle{\norm{\alpha}_\infty = \max_{1\leq k\leq n}\{\alpha_k\}}$
be the largest element of the tuple $\alpha\in  \rl^n$, and let
\begin{equation}
    \label{eq-malpha}
    m(\alpha)=\#{\{j: \alpha_j= \norm{\alpha}_\infty\}}
\end{equation}
be the number of largest elements in the tuple $\alpha$.
We will prove the following, where in \eqref{eq-poyest1} and \eqref{eq-polyest2}, $\rho_\alpha$ and $\rho_{2\alpha}$ stand for the function introduced in \eqref{eq-rhodef}.

\begin{thm}
    \label{thm-polydisc}
    Let $\alpha=(\alpha_1,\dots,\alpha_n)$ be a tuple of
    nonnegative numbers such that $\alpha\not=0$.
    Set 
       \begin{equation}
          \label{eq-pstar}
          p_*= \frac{2 \norm{\alpha}_\infty+2}{\norm{\alpha}_\infty+2}.
      \end{equation}
          \begin{enumerate}
        \item There is a constant $C>0$ such that for each measurable subset $F\subset \D^n$:
        \begin{equation}
            \label{eq-poyest1}\norm{  P^+_{\D^n}(\rho_\alpha\cdot 1_F)}_{L^{p_*}(\D^n)}^{p_*} \leq C 
        \int_F \rho_{2\alpha}(-\log \rho_\alpha)^{(p_*-1)(m(\alpha)-1)}\,dV.
        \end{equation}
        \item For $p\in (p_*,2]$ there is a constant $C_p$ such that for each measurable subset $F\subset \D^n$:
        \begin{equation}
            \label{eq-polyest2}
            \norm{  P^+_{\D^n}(\rho_\alpha\cdot 1_F)}_{L^{p}(\D^n)}^{p} \leq C_p
        \int_F \rho_{2\alpha}\,dV.
        \end{equation}
    \end{enumerate}
\end{thm}
\subsection{Preliminary computations}
We will use the following standard notation. Given real valued functions $A,B$ on a set $\mathfrak{X},$ by
$A\lesssim B$ we mean that there is a constant $C>0$ such that $A(x)\leq C B(x)$ for each $x\in \mathfrak{X}.$ The notation $A\approx B$ means that both $A\lesssim B, B \lesssim A$ hold. As usual, we will suppress mentioning the precise set $\mathfrak{X}$ but it should be clear from the context. 
\begin{lem}\label{lem-ualpha}
Let $\alpha=(\alpha_1,\dots,\alpha_n)$ with $\alpha_j\geq 0$ for each $j$ and $\alpha\not=0$. Then for $0<s\leq 1$
\begin{equation}
    \label{eq-ualpha}\abs{\{z\in\mathbb D^n:\rho_\alpha(z)<s\}}
    \approx s^{\frac{2}{\norm{\alpha}_\infty}}\cdot\left(-\log{s}\right)^{m(\alpha)-1},
\end{equation}
where the implied constants in \eqref{eq-ualpha}
are independent of $s$,  but  depend on $\alpha$.
\end{lem}
  \begin{proof} For simplicity of notation set 
  $U_\alpha(s):=\{z\in\mathbb D^n:\rho_\alpha(z)<s\}.$
  We proceed by induction on the dimension
  $n$.  If $n=1$, then $\alpha=(\alpha_1)$ with $\alpha_1>0$, $\norm{\alpha}_\infty=\alpha_1$, $m(\alpha)=1$, and  $U_\alpha(s)$ is the disc $\left\{\abs{z}<s^{\frac{1}{\alpha_1}}\right\}$ in the plane. Consequently
  $ \abs{U_\alpha(s)}= \pi s^{\frac{2}{\alpha_1}}$,
  establishing the result. 
  
   Suppose the estimate \eqref{eq-ualpha}  
   holds for some $n$.
 To show it for $n+1$, let $\alpha=(\alpha_1,\dots, \alpha_{n+1})$ be an $(n+1)$-tuple of nonnegative
 numbers for which we want to show the analog of \eqref{eq-ualpha} with $n$ replaced by $n+1$. We can assume without loss of generality that 
 $ \alpha_{n+1}\geq \alpha_{n}\geq \dots \geq \alpha_1,$
 so that 
 $ \norm{\alpha}_\infty=\alpha_{n+1}=\alpha_{n}=\dots= \alpha_{n+2-m(\alpha)},$
and $\alpha_k < \norm{\alpha}_\infty$ if $k<n+2-m(\alpha).$ 
Let $V$ be the polydisc:
\[ V=\{z\in \D^{n+1}: \abs{z_{n+1}}^{\alpha_{n+1}}<s\}=\D^n\times \left\{z\in \cx: \abs{z}< s^{\frac{1}{\norm{\alpha}_\infty}}\right\},\]
so that
\begin{equation}
    \label{eq-v1}
    \abs{V}=\pi^n s^{\frac{2}{\norm{\alpha}_\infty}}.
\end{equation} 
Write $z=(z',z_{n+1})\in \cx^n\times \cx$ and
 $\alpha=(\alpha',\alpha_{n+1})\in \rl^n\times \rl.$
Then,
\begin{align}
   U_\alpha(s)\setminus V &= \left\{z\in \D^{n+1}:s \leq\abs{z_{n+1}}^{\alpha_{n+1}}<1, \quad\rho_{\alpha'}(z') < \frac{s}{\abs{z_{n+1}}^{\alpha_{n+1}}} \right\}\nonumber\\
    &=\left\{z\in \D^{n+1}:s \leq\abs{z_{n+1}}^{\alpha_{n+1}}<1, \quad z'\in U_{\alpha'}\left(\frac{s}{\abs{z_{n+1}}^{\alpha_{n+1}}} \right) \right\}\label{eq-v2def}.
\end{align}
By the induction hypothesis
\begin{equation}
    \label{eq-induction}
     \abs{U_{\alpha'}(t)}=\abs{\{z\in \D^n: \rho_{\alpha'}(z_1,\dots, z_n)<t \}}\approx t^{\frac{2}{\norm{\alpha'}_\infty}}\left( -\log {t}\right)^{m(\alpha')-1}.
\end{equation}
From \eqref{eq-v2def}:
 \begin{align}
     \abs{U_\alpha(s)\setminus V }&= \int_{\{s \leq\abs{z_{n+1}}^{\alpha_{n+1}}<1\}}\abs{U_{\alpha'}\left(\frac{s}{\abs{z_{n+1}}^{\alpha_{n+1}}} \right)}dV(z_{n+1})\nonumber\\
     &=2\pi \int_{s^{\frac{1}{\alpha_{n+1}}}}^1 \abs{U_{\alpha'}\left(\frac{s}{{r}^{\alpha_{n+1}}} \right)}rdr\\
     &\approx \int_{s^{\frac{1}{\alpha_{n+1}}}}^1 
     \left(\frac{s}{r^{\alpha_{n+1}}} \right)^{\frac{2}{\norm{\alpha'}_\infty}} \left( \log\left( \frac{r^{\alpha_{n+1}}}{s}\right)\right)^{m(\alpha')-1}rdr& \text{ using \eqref{eq-induction}}\nonumber\\
     &\approx s^{\frac{2}{\alpha_n}} \int_{s^{\frac{1}{\alpha_{n+1}}}}^1 r^{1-2\frac{\alpha_{n+1}}{\alpha_n}}\left( \log\left( \frac{r^{\alpha_{n+1}}}{s}\right)\right)^{m(\alpha')-1}dr.\label{eq-maincomp}
 \end{align}
We now consider several cases:
\begin{enumerate}[wide]
    \item $m(\alpha)=m(\alpha')=1$. Then we have
    $\alpha_{n+1}>\alpha_{n}$ and if $n\geq 2$, $\alpha_n>\alpha_{n-1}.$ Notice that we have then
    $1-2\frac{\alpha_{n+1}}{\alpha_n}<-1,$
    so 
    \begin{align*}
        \abs{U_\alpha(s)\setminus V }&\approx s^{\frac{2}{\alpha_n}} \int_{s^{\frac{1}{\alpha_{n+1}}}}^1\,r^{1-2\frac{\alpha_{n+1}}{\alpha_n}}dr\\
        &=\frac{\alpha_n}{2(\alpha_{n+1}-\alpha_n)}\left(s^{\frac{2}{\alpha_{n+1}}}-s^{\frac{2}{\alpha_n}}\right),& \text{ by a direct evaluation}\\
        &\approx s^{\frac{2}{\alpha_{n+1}}}= s^{\frac{2}{\norm{\alpha}_\infty}}.
    \end{align*}
    Combining this with \eqref{eq-v1}, the result follows in this case.

\item $m(\alpha)=1, m(\alpha')>1$.
Let $\beta\in \rl^{n+1}$ be the multi-index such that $\beta_{n+1}=\alpha_{n+1}$ and $\beta_j=0$ if $1\leq j \leq n$, so that we have $\beta_j \leq \alpha_j$ for $1\leq j \leq n+1$.
Also let $\gamma=(\gamma',\gamma_{n+1})\in \rl^{n+1}$ be a tuple 
of nonnegative numbers, with $\gamma'\in \rl^n$ and $\gamma_{n+1}\in \rl$ such that $\gamma_{n+1}=\alpha_{n+1}$, $m(\gamma)=m(\gamma')=1$ and
for each $1\leq j\leq n$ we have $\alpha_{n+1}>\gamma_j> \alpha_j$,
so that for each $1\leq j \leq n+1$, we have $\gamma_j \geq \alpha_j.$ It is easy to see that such a $\gamma$ exists. For $z\in \D^n$, then we have $\rho_\gamma(z) \leq \rho_\alpha(z) \leq \rho_\beta(z)$, so that for $0<s\leq 1$ we have the inclusions:
\[ U_\beta(s)\subseteq U_\alpha(s) \subseteq U_\gamma(s).\]

Since $U_\beta(s) = \{z\in \D^n: \abs{z}\leq s^{\frac{1}{\alpha_{n+1}}}\}$, we see that $ \abs{U_\beta(s)}= \pi^n s^{\frac{2}{\alpha_{n+1}}}.$ Since $m(\gamma)=m(\gamma')=1$ and $\norm{\gamma}_\infty= \alpha_{n+1}$, it follows by the previous case that
$ \abs{U_\gamma(s)}\approx s^{\frac{2}{{\alpha_{n+1}}}}.$ Therefore, $\abs{U_\beta(s)}\approx s^{\frac{2}{{\alpha_{n+1}}}}
=s^{\frac{2}{\norm{\alpha}_\infty}}\cdot\left(\log\left( \frac{1}{s}\right)\right)^{m(\alpha)-1}$, since $m(\alpha)=1$.

\item $m(\alpha)>1$. Then $\alpha_{n}=\alpha_{n+1}$ and $m(\alpha')=m(\alpha)-1$ so we have
\begin{align*}
    \abs{U_\alpha(s)\setminus V}&\approx \eqref{eq-maincomp}=s^\frac{2}{\alpha_n} \int_{s^{\frac{1}{\alpha_{n+1}}}}^1 \left(\log\left(\frac{r^{\alpha_{n+1}}}{s} \right) \right)^{m(\alpha)-2} \frac{dr}{r}\\
    &=\sigma^2\cdot \alpha_{n+1}^{m(\alpha)-2} \int_{\sigma}^1 \left(\log\left(\frac{r}{\sigma} \right)\right)^{m(\alpha)-2} \frac{dr}{r}& \text{ where } \sigma =s^{\frac{1}{\alpha_{n}}}= s^{\frac{1}{\alpha_{n+1}}}\\\
    &= \sigma^2\cdot \alpha_{n+1}^{m(\alpha)-2}\int_1^{\frac{1}{\sigma}}(\log t)^{m(\alpha)-2} \frac{dt}{t}&\text{ where } t= \frac{r}{\sigma}\\
    &\approx \sigma^2 \left(\log\left(\frac{1}{\sigma} \right) \right)^{m(\alpha)-1} & \text{ by direct evaluation using } u=\log t\\
    & \approx s^{\frac{2}{\alpha_{n+1}}}\left(-\log {s}  \right)^{m(\alpha)-1},
    \end{align*}
completing the induction and the proof of the lemma.
\end{enumerate}  
\end{proof}
The next lemma is crucial for us to establish H\"older type estimates in Proposition \ref{prop-key}.
\begin{lem}\label{lem-gamma}
    Let $0<\gamma<1$ and let $\mu\geq 0$ be an integer. Then there is a $C>0$ such that for each subset $I\subseteq\mathbb N\bigcup \{0\}$, there is a
    $j_*\in I$ satisfying:
    \[ \sum_{j\in I}j^\mu \gamma^j \leq C j_*^\mu \gamma^{j_*}.\]
\end{lem}
\begin{proof}
 Set $c_j=j^\mu\gamma^j$, and
 let $j_0$ denote the largest integer $j$ satisfying 
\[\frac{c_{j+1}}{c_j}=\left(\frac{j+1}{j}\right)^{\mu}\gamma\geq 1.\]
Then, $c_{j_0+1}=\max\{c_j:j\geq 0\}$. Furthermore,
$\frac{c_{j+1}}{c_j}\geq \frac{c_{j_0}}{c_{j_0-1}}>1$ for all $j\leq j_0-1$ and $\frac{c_{j+1}}{c_j}\leq \frac{c_{j_0+2}}{c_{j_0+1}}<1$ for all $j\geq j_0+1$. Set $j_{1}=\max\{j\in I:j\leq j_0+1\}$ and $j_{2}=\min\{j\in I:j\geq j_0+1\}$.
Then the inequalities for the ratio $c_{j+1}/c_j$ imply that
\begin{align*}
   \sum_{j\in I,j\leq j_0+1}j^\mu \gamma^j\approx c_{j_{1}} \quad \text{ and }
    \sum_{j\in I_k,j\geq j_0+1}j^\mu\gamma^j\approx c_{j_{2}}.
\end{align*}
Choose $j_*$ from $j_{1}$ and $j_{2}$ such that $c_{j_*}=\max\{c_{j_{1}},c_{j_{2}}\}$. Then the claimed estimate follows:
\[\sum_{j\in I}j^\mu \gamma^j\approx c_{j_{1}}+c_{j_{2}}\approx c_{j_*}.\] 
\end{proof}
\begin{prop}\label{prop-key}
      Let $t>1$ and $\alpha=(\alpha_1,\dots,\alpha_n)$ with $\alpha_j\geq 0$ for all $j$ and $\alpha\not=0$. Set
       \begin{equation}
          \label{eq-p}
          p= \frac{t \norm{\alpha}_\infty+2}{\norm{\alpha}_\infty+2}.
      \end{equation}
      
      \begin{enumerate}
    \item  There is a $C>0$ such that for each measurable subset $E\subset \D^n$ we have
    \begin{equation}
        \label{eq-conc1}
        \int_E \rho_\alpha \,dV\leq C\left(\int_{E}\rho_{t\alpha}\cdot \left(-\log{\rho_\alpha }\right)^{(p-1)(m(\alpha)-1)}dV\right)^{1/p}.
    \end{equation}
 \item For each $\epsilon>0$ there is a $C_\epsilon>0$ such that for each measurable subset $E\subset \D^n$ we have
 \begin{equation}
     \label{eq-conc2}
      \int_E\rho_{\alpha}dV\leq C_\epsilon\left(\int_{E}\rho_{t\alpha}dV\right)^{1/(p+\epsilon)}.
 \end{equation}
    \end{enumerate}
    \end{prop}
\begin{proof} 
  
	For an integer $j\geq 0$, if we set $\displaystyle{U_j=\left\{z\in \D^n: \frac{1}{2^j}\leq \rho_\alpha(z) < \frac{1}{2^{j-1}} \right\}},$
    then from  Lemma~\ref{lem-ualpha}, we see that
    \begin{equation}
        \label{eq-voluj}
        \abs{U_j}\approx  2^{-\frac{2j}{\norm{\alpha}_\infty}}\cdot j^{m(\alpha)-1}.
    \end{equation}
    Let $\eta\geq 0$. On the set $U_j$, we clearly have
    \[ \rho_{t\alpha}(-\log  \rho_\alpha)^\eta \approx j^\eta 2^{-tj},\]
    where the implied constant is independent of $j$ but depends on $\alpha$ and $t$.
    If for $k\geq 1$, we set
    $\displaystyle{ I_k=\left\{ j \geq 1: \frac{1}{2^k}< \frac{\abs
{E\cap U_j}}{\abs{U_j}} \leq \frac{1}{2^{k-1}}\right\}}$, we have
 \begin{align}
      \int_E \rho_{t\alpha}\log(-\rho_\alpha)^\eta\,dV&= \sum_{j=1}^\infty \int_{E\cap U_j}\rho_{t\alpha}\log(-\rho_\alpha)^\eta \,dV
\approx \sum_{j=1}^\infty j^\eta 2^{-tj}\abs{E\cap U_j}\nonumber\\
&= \sum_{j=1}^\infty j^\eta 2^{-tj} \abs{U_j}\cdot \frac{\abs{E\cap U_j}}{\abs{U_j}}=\sum_{k=1}^\infty\sum_{j\in I_k} j^\eta 2^{-tj} \abs{U_j}\cdot \frac{\abs{E\cap U_j}}{\abs{U_j}}\nonumber\\
&= \sum_{k=1}^\infty\left( \sum_{j\in I_k}  j^\eta 2^{-tj}2^{-\frac{2j}{\norm{\alpha}_\infty}}\cdot j^{m(\alpha)-1} \right)\frac{1}{2^k}\nonumber\\
&= \sum_{k=1}^\infty\left( \sum_{j\in I_k}   j^{\eta+m(\alpha)-1}  2^{-j\left(t+\frac{2}{\norm{\alpha}_\infty}\right)}\right)\frac{1}{2^k}.\label{eq-comp1}
  \end{align}
Thanks to Lemma~\ref{lem-gamma},  for each $k$, there exists a $j_k\in I_k$ such that
\[\sum_{j\in I_k}j^{\eta+m(\alpha)-1}  2^{-j\left(t+\frac{2}{\norm{\alpha}_\infty}\right)}\approx j_k^{\eta+m(\alpha)-1}  2^{-j_k\left(t+\frac{2}{\norm{\alpha}_\infty}\right)},\]
where the implied constant depends only on $\alpha$ and $\eta$ (and not on $k$). Therefore
\begin{equation}
    \label{eq-rhotalpha}
     \int_E \rho_{t\alpha}\log(-\rho_\alpha)^\eta\, dV \approx  \eqref{eq-comp1} \approx \sum_{k=1}^\infty\left(  j_k^{\eta+m(\alpha)-1}  2^{-j_k\left(t+\frac{2}{\norm{\alpha}_\infty}\right)}\right)\frac{1}{2^k}.
\end{equation}
Taking $t=1$ and $\eta=0$ in \eqref{eq-rhotalpha} we see that
\begin{align}
    \int_E \rho_\alpha dV &\approx  \sum_{k=1}^\infty\left(  j_k^{m(\alpha)-1}  2^{-j_k\left(1+\frac{2}{\norm{\alpha}_\infty}\right)}\right)\frac{1}{2^k}\label{eq-approx}\\
    &\leq \left(\sum_{k=1}^\infty\left(  j_k^{m(\alpha)-1}  2^{-j_k\left(1+\frac{2}{\norm{\alpha}_\infty}\right)}\right)^p\frac{1}{2^k} \right)^\frac{1}{p} \label{eq-jensen}\\
    &= \left(\sum_{k=1}^\infty  j_k^{p(m(\alpha)-1)}  2^{-pj_k\left(1+\frac{2}{\norm{\alpha}_\infty}\right)}\frac{1}{2^k} \right)^\frac{1}{p}.\label{eq-comp2}
\end{align}
Here in \eqref{eq-jensen}, we have use Jensen's inequality $\phi\left(\int f d\mu\right) \leq \int\phi\circ f\, d\mu$, where $\phi$ is the convex function $\phi(t)=t^p$ with $p$ as in \eqref{eq-p}, the measure $\mu$
on the set of positive integers is such that $\mu(k)=\frac{1}{2^k}$ and $f(k)=j_k^{m(\alpha)-1}  2^{-j_k\left(1+\frac{2}{\norm{\alpha}_\infty}\right)}. $ Taking $\eta=(p-1)(m(\alpha)-1)$
in \eqref{eq-rhotalpha}, and noticing that \eqref{eq-p} is equivalent to $\displaystyle{t+\frac{2}{\norm{\alpha}_\infty}}=p\left(1 +\frac{2}{\norm{\alpha}_\infty} \right)$,
we see that
\[ \int_{E}\rho_{t\alpha}\cdot \left(-\log{\rho_\alpha }\right)^{(p-1)(m(\alpha)-1)}dV \approx \sum_{k=1}^\infty  j_k^{p(m(\alpha)-1)}  2^{-pj_k\left(1+\frac{2}{\norm{\alpha}_\infty}\right)}\frac{1}{2^k}.\]
Comparing the right hand side with \eqref{eq-comp2} the inequality \eqref{eq-conc1} follows.

To see \eqref{eq-conc2}, let $\delta>0$ be defined by the relation $\displaystyle{(p+\epsilon)\left(1+\frac{2}{\norm{\alpha}_\infty+\delta}\right)=t+\frac{2}{\norm{\alpha}_\infty}}$.
Since we have $$j^{m(\alpha)-1} 2^{-\frac{j}{\norm{\alpha}_\infty}} \lesssim 2^{-\frac{j}{\norm{\alpha}_\infty+\delta}}$$
(with the implied constant depending on $\epsilon$ and $\alpha$ but not $j$), from \eqref{eq-approx} we see that
\begin{align}
      \int_E \rho_\alpha dV &\lesssim \sum_{k=1}^\infty 2^{-j_k\left(1+\frac{2}{\norm{\alpha}_\infty +\delta} \right)}\cdot \frac{1}{2^k}\nonumber\\
      & \leq \left(\sum_{k=1}^\infty 
      \left(2^{-j_k\left(1+\frac{2}{\norm{\alpha}_\infty +\delta} \right)}\right)^{(p+\epsilon)}\cdot \frac{1}{2^k}\right)^\frac{1}{p+\epsilon}\label{eq-jensen2}\\
      & =   \left(\sum_{k=1}^\infty 
      \left(2^{-j_k\left(t+\frac{2}{\norm{\alpha}_\infty}\right)}\right)\cdot \frac{1}{2^k}\right)^\frac{1}{p+\epsilon}\nonumber\\
      & \leq \left(\sum_{k=1}^\infty\left(  j_k^{m(\alpha)-1}  2^{-j_k\left(t+\frac{2}{\norm{\alpha}_\infty}\right)}\right)\frac{1}{2^k}\right)^\frac{1}{p+\epsilon}\nonumber\\
      & \approx \left(\int_E\rho_{t\alpha} dV\right)^\frac{1}{p+\epsilon} &\text{ using \eqref{eq-rhotalpha}},
\end{align}
where again in \eqref{eq-jensen2} we have used the Jensen inequality, this time with the convex function $t\mapsto t^{p+\epsilon}.$ This completes the proof of \eqref{eq-conc2}.
\end{proof}
\subsection{A special case of the main estimate}
Denote by $\D_{\frac{1}{2}}^n$ the polydisc of 
radius $\frac{1}{2}$:
\[\D_{\frac{1}{2}}^n=\left\{z\in\cx^n: \abs{z_j}<\frac{1}{2} \text{ for } 1\leq j \leq n\right\}. \]
\begin{prop}
    \label{prop-specialcase}
    Let $0\leq \ell \leq n$, and let $G$ be a measurable subset of the product domain
    \begin{equation}
        \label{eq-fl}
       \D^\ell_{\frac{1}{2}} \times \left(\D^{n-\ell}\setminus \D^{n-\ell}_{\frac{1}{2}}\right)=\left\{z\in \D^n:\abs{z_j}< \frac{1}{2} \text{ for } j\leq \ell, \text{ and }\abs{z_j}\geq \frac{1}{2} \text{ for } j \geq \ell+1 \right\}. 
    \end{equation}
    If $\beta=(\beta_1,\dots, \beta_n)$ is a nonzero 
    tuple of nonnegative numbers, then we have
    \begin{equation}
        \label{eq-specialcase1}
          \norm{P^+_{\mathbb{D}^n}(\rho_\beta\cdot1_{G})}^{p_*}_{L^{p_*}(\mathbb D^n)}\lesssim \int_{G}\rho_{2\beta}\cdot(-\log \rho_{\beta})^{(p_*-1)(m(\beta)-1)}\,dV,
    \end{equation}
 where $\displaystyle{ p_*=  \frac{2 \norm{\beta}_\infty+2}{\norm{\beta}_\infty+2}.}$ Further, if $p\in (p_*,2]$ we have

 \begin{equation}
     \label{eq-specialcase2}
 \norm{P^+_{\mathbb{D}^n}(\rho_\beta\cdot1_{G})}^{p}_{L^{p}(\mathbb D^n)}\lesssim \int_{G}\rho_{2\beta}\,dV.  \end{equation}
\end{prop}
\begin{proof} 
Write the coordinates of $\cx^n$ as $(z',z'')$ where $z'\in \cx^\ell$ and $z''\in \cx^{n-\ell},$ and write the tuple $\beta$ as $(\beta', \beta'')$ 
where $\beta'\in \rl^\ell$ and $\beta''\in \rl^{n-\ell}.$

For $w\in G$ and $z\in \D^n$, we have $\abs{w_j}<\frac{1}{2}$ and $\abs{z_j}<1$, so $\abs{1-z_j \ol{w_j}}> \frac{1}{2}$ whenever $j \leq \ell$. Therefore 
\[\abs{K_{\D^\ell}(z',w')}= \frac{1}{\pi^\ell} \cdot \frac{1}{\prod_{j=1}^\ell \abs{1- z_j \ol{w_j}}^2} \leq \left(\frac{4}{\pi} \right)^\ell,\]
so that \begin{equation}
    \label{eq-inp1}
    \abs{K_{\D^n}(z,w)}= \abs{K_{\D^\ell}(z',w')}\cdot \abs{K_{\D^{n-\ell}}(z'',w'')} \lesssim \abs{K_{\D^{n-\ell}}(z'',w'')}.
\end{equation}
Moreover, for $w\in G$ we have $\frac{1}{2}\leq \abs{w_j}<1$ if $j\geq \ell+1,$ so
\begin{equation}
    \label{eq-inp2}
    \rho_\beta(w)=\rho_{\beta'}(w')\rho_{\beta''}(w'')\approx \rho_{\beta'}(w').
\end{equation}
Therefore for $z=(z',z'')\in \D^n$
\begin{align}
    P^+_{\D^n}(\rho_\beta 1_G)(z',z'')&= \int_{\D^n}\abs{K_{\D^n}(z,w)}\rho_\beta(w)1_G(w)dV(w)\nonumber\\
    &\lesssim \int_{\D^n} \abs{K_{\D^{n-\ell}}(z'',w'')} \rho_{\beta'}(w')1_{G}(w',w'')dV \quad \text{ using \eqref{eq-inp1} and \eqref{eq-inp2}}\nonumber\\
    &= \int_{\D^{n-\ell}}\abs{K_{\D^{n-\ell}}(z'',w'')}\left(\int_{\D^\ell}\rho_{\beta}(w')1_G(w',w'')dV(w') \right) dV(w'')\nonumber\\
    &=P_{\D^{n-\ell}}^+(h)(z''), \label{eq-comp7}
 \end{align}
where $h$ is the function on $\D^{n-\ell}$ given by
\begin{equation}\label{eq-h}
    h(w'')=\int_{\D^\ell}\rho_{\beta}(w')1_G(w',w'')dV(w')
    =\int_{G(w'')}\rho_{\beta'}dV,
\end{equation}
with \[G(w'')= \{w'\in \D: (w',w'')\in G\}.\]
For $p>1$ we have
    \begin{align}
        \norm{P_{\D^n}^+(\rho_\beta 1_G)}_{L^{{p}}(\D^n)}^{{p}} & \lesssim \norm{P^+_{\D^{n-\ell}}(h)}^{{p}}_{L^{{p}}(\D^{n-\ell})} \quad \text{ using \eqref{eq-comp7}}\nonumber\\
        & \lesssim \norm{h}^{{p}}_{L^{{p}}(\D^{n-\ell})}\nonumber \quad \text{ since   $P^+_{\D^{n-\ell}}$ is bounded in $L^{{p}}(\D^{n-\ell})$ }\nonumber\\
        &= \int_{\D^{n-\ell}} (h(w''))^{{p}}dV(w'')\nonumber\\
        &= \int_{\D^{n-\ell}} \left( \int_{G(w'')}\rho_{\beta'}(w')dV(w')\right)^{{p}}dV(w'').\label{eq-comp8}
        \end{align}
To prove \eqref{eq-specialcase1}, we consider two cases, depending on the tuple $\beta$:
\begin{enumerate}[wide]
    \item Case 1: The $\ell$-tuple $\beta'=(\beta_1,\dots, \beta_\ell)$ has an element equal to $\norm{\beta}_\infty= \max\{\beta_k: 1\leq k \leq n\},$ i.e. $\norm{\beta'}_\infty= \norm{\beta}_\infty$, where $\norm{\beta'}_\infty=\max\{\beta_k, 1 \leq k \leq \ell\}$.
    
    In this case,
    \[ \frac{2\norm{\beta'}_\infty+2}{\norm{\beta'}_\infty+2}=p_*= \frac{2\norm{\beta}_\infty+2}{\norm{\beta}_\infty+2},\]
    so by conclusion~\eqref{eq-conc1} of Proposition~\ref{prop-key} (with $t=2$), we have for each $w''\in \D^{n-\ell}:$
    \begin{align}
        (h(w''))^{p_*}= \left( \int_{G(w'')}\rho_{\beta'}dV\right)^{p_*}
        & \lesssim \int_{G(w'')} \rho_{2\beta'}\cdot (-\log \rho_{\beta'})^{(p_*-1)(m(\beta')-1)}dV.      \label{eq-comp5}
    \end{align}
    Then taking $p=p_*$ in \eqref{eq-comp8} and using \eqref{eq-comp5} we obtain
   \begin{align}  \norm{P_{\D^n}^+(\rho_\beta 1_G)}_{L^{p_*}(\D^n)}^{p_*}& \lesssim \eqref{eq-comp8}\nonumber\\
       &\lesssim \int_{\D^{n-\ell}}\int_{G(w'')} \rho_{2\beta'}(w')\cdot (-\log \rho_{\beta'}(w'))^{(p_*-1)(m(\beta')-1)}dV(w') dV(w'') \nonumber\\
        &=\int_{G}\rho_{2\beta'}(w')\cdot (-\log \rho_{\beta'}(w'))^{(p_*-1)(m(\beta')-1)}dV(w).\label{eq-comp6}
    \end{align}
    The integrand in \eqref{eq-comp6} is dominated by the function
    \begin{align*}
       w=(w',w'')&\mapsto \rho_{2\beta'}(w')\rho_{2\beta''}(w'')\cdot (-\log \rho_{\beta'}(w')-\log\rho_{\beta''}(w''))^{(p_*-1)(m(\beta)-1)}\\
       &=\left(\rho_{2\beta}\cdot(-\log \rho_{\beta})^{(p_*-1)(m(\beta)-1)}\right)(w),
    \end{align*}
    since $\rho_{2\beta''}\approx 1$, $-\log\rho_{\beta''}>0$ and $m(\beta)\geq m(\beta').$ This completes the proof of \eqref{eq-specialcase1} in this case. 
    \item Case 2: no element of the $\ell$-tuple $\beta'=(\beta_1,\dots, \beta_\ell)$ is equal to $\norm{\beta}_\infty$. Then $\norm{\beta'}_\infty< \norm{\beta}_\infty$, and if we set 
    $\displaystyle{p'= \frac{2\norm{\beta'}_\infty+2}{\norm{\beta'}_\infty+2}}$, then we have $p'<p_*$.

Now, using the conclusion \eqref{eq-conc2} of Proposition~\ref{prop-key} (with $t=2$ and $p'+\epsilon=p_*$) we conclude that for each $w''\in \D^{n-\ell}$
\begin{equation}
     (h(w''))^{p_*}= \left( \int_{G(w'')}\rho_{\beta'}dV\right)^{p_*}
         \lesssim \int_{G(w'')} \rho_{2\beta'}\,dV.      \label{eq-comp9}
\end{equation}
Therefore, again taking $p=p_*$ in \eqref{eq-comp8} gives
 \begin{align}  \norm{P_{\D^n}^+(\rho_\beta 1_G)}_{L^{p_*}(\D^n)}^{p_*} \lesssim \eqref{eq-comp8}&\lesssim \int_{\D^{n-\ell}}\int_{G(w'')} \rho_{2\beta'}(w')dV(w') dV(w'') \nonumber\\
 & = \int_{G} \rho_{2\beta'}(w')dV(w)\nonumber\\
 & \lesssim \int_{G} \rho_{2\beta}\cdot(-\log \rho_{\beta})^{(p_*-1)(m(\beta)-1)}dV.\nonumber
   \end{align}    
    Here we use the facts that $\rho_{2\beta''}\approx 1$ and that
    \[ -\log \rho_{\beta}(w)=-\log \rho_{\beta'}(w')- \log \rho_{\beta''}(w'')\gtrsim 1,\]
    since  $- \log \rho_{\beta''}(w'') \gtrsim 1$ and
    $-\log \rho_{\beta'}(w')>0$ if $w=(w',w'')\in G.$
    This completes the second case, and also the proof
    of \eqref{eq-specialcase1}.
\end{enumerate}

The proof of \eqref{eq-specialcase2} is similar to that of case (2) above. Since $p>p_*$, by  conclusion \eqref{eq-conc2} of Proposition~\ref{prop-key} (with $t=2$) we conclude that for each $w''\in \D^{n-\ell}$
\begin{equation}
     (h(w''))^{p}= \left( \int_{G(w'')}\rho_{\beta'}dV\right)^{p}
         \lesssim \int_{G(w'')} \rho_{2\beta'}\,dV.      \label{eq-comp10}
\end{equation}
Therefore, by \eqref{eq-comp8}
 \begin{align}  \norm{P_{\D^n}^+(\rho_\beta 1_G)}_{L^{p}(\D^n)}^{p} \lesssim \eqref{eq-comp8}&\lesssim \int_{\D^{n-\ell}}\int_{G(w'')} \rho_{2\beta'}(w')dV(w') dV(w'') \nonumber\\
 & = \int_{G} \rho_{2\beta'}(w')dV(w)\nonumber \lesssim \int_{G} \rho_{2\beta}dV, \nonumber
   \end{align}    
    where we use the fact that $\rho_{2\beta''}\approx 1$ on $G$. This completes the proof of \eqref{eq-specialcase2}.
\end{proof}
\subsection{Proof of Theorem~\ref{thm-polydisc}}
    Let $\mathcal S$ denote the power set of the set $\{1,2,\dots,n\}$. For a subset $\mathcal J=\{j_1,\dots,j_k\}\in\mathcal S$, define a projection
    map $\pi_{\mathcal{J}}:\cx^n\to \cx^{\abs{\mathcal{J}}}$ by setting $\pi_{\mathcal J}(z)=(z_{j_1},\dots,z_{j_k})$ for $z\in \cx^n$.
    Denoting by $\J^c$ the complement $\{1,2,\dots, n\}\setminus \J$, we abuse notation slightly by 
    writing $z=(\pi_\J(z), \pi_{\J^c}(z)).$
    Notice also $\pi_\J(\D^n)\cong \D^{\#\J}.$ 
    
    Define the operator $P^+_{\mathcal J}$  as follows.
    For a function $f$ on $\D^n,$ we set for $z\in \D^n$:
 \[P^+_{\mathcal J}(f)(z)=\int\limits_{\pi_{\mathcal J}(\mathbb {D}^n)}\frac{f\left(\pi_{\J}(w),\pi_{\J^c}(z)\right)dV(\pi_{\mathcal J}(w))}{\prod\limits_{j\in\mathcal J}\pi\abs{1-z_j\bar w_j}^2},\]
where, by slight abuse of notation, we mean by $\left(\pi_{\J}(w),\pi_{\J^c}(z)\right)$ the $n$-tuple in which the $j$-th element is $w_j$ if $j\in \J$ and is 
$z_j$ if $j\not\in \J$. Therefore, we can think of 
$P^+_{\mathcal J}$ as the positive Bergman operator 
acting in the variables $z_j, j\in \J$, or more precisely  as a tensor product of the positive Bergman operator on a polydisc of dimension $\#\J$ with 
the identity operator on another polydisc of dimension 
$\#\J^c=n-\#\J.$

For $\mathcal J\in \mathcal S$, we set
    \[F_{\mathcal J}=\left\{w\in F:|w_j|<\frac{1}{2} \text{ for  }j\in \mathcal J \text{ and }\frac{1}{2}\leq |w_k|<1 \text{ for  }k\in \mathcal J^{c}\right\}.\]
The collection $\left\{F_{\mathcal J}, \J\in \mathcal{S}\right\}$  consists of pairwise disjoint sets and $F=\bigcup\limits_{\mathcal J\in\mathcal S}F_{\mathcal J}$. Therefore 
 $1_F=\sum\limits_{\mathcal J\in\mathcal S}1_{F_{\mathcal J}}$ and consequently 
 \[ \norm{P^+_{\mathbb{D}^n}(\rho_\alpha\cdot1_F)}_{L^{p_*}(\mathbb D^n)}\leq\sum_{\mathcal J\in\mathcal S}\norm{P^+_{\mathbb D^n}(\rho_\alpha\cdot1_{F_{\mathcal J}})}_{L^{p_*}(\mathbb D^n)} \lesssim \left( \sum_{\mathcal J\in\mathcal S}\norm{P^+_{\mathbb D^n}(\rho_\alpha\cdot1_{F_{\mathcal J}})}_{L^{p_*}}^{p_*}\right)^{\frac{1}{p_*}}.\]

 Therefore, to complete the proof of \eqref{eq-poyest1},
 it is enough to show  that 
 for each $\mathcal J\in \mathcal S$, we have
 \begin{equation}\label{eq-FJ}
     \norm{P^+_{\mathbb{D}^n}(\rho_\alpha\cdot1_{F_{\mathcal J}})}^{p_*}_{L^{p_*}(\mathbb D^n)}\lesssim \int_{F_{\mathcal J}}\rho_{2\alpha}(-\log \rho_{\alpha})^{(p_*-1)(m(\alpha)-1)}\,dV.
 \end{equation}
Similarly, to complete the proof of \eqref{eq-polyest2},
we need to show that for each $p\in (p_*,2]$ we have
\begin{equation}
    \label{eq-FJ2}\norm{P^+_{\mathbb D^n}(\rho_\alpha\cdot1_{F_{\mathcal J}})}^{p}_{L^p(\mathbb D^n)}\lesssim \int_{F_{\mathcal J}}\rho_{2\alpha}\,dV.
\end{equation}

 Let $\ell= \#\J$ be the cardinality of $\J$,  and let $\sigma:\{1,2,\dots,n\}\to \{1,2,\dots,n\}$ be a permutation such that $\sigma(\J)=\{1,2,\dots,\ell\}$.
 Then  $\sigma(\J^c)=\{\ell+1,\dots, n\}$,  and define
 a unitary automorphism $S:\cx^n\to \cx^n$ by setting
 $S(z_1,\dots, z_n)=(z_{\sigma(1)},\dots, z_{\sigma(n)}).$ Making a change of variables by 
 the map $S$ on both sides of \eqref{eq-FJ}, we have 
 the equivalent inequality:
 \begin{equation}
     \label{eq-G}
     \norm{P^+_{\mathbb{D}^n}(\rho_\beta\cdot1_{G})}^{p_*}_{L^{p_*}(\mathbb D^n)}\lesssim \int_{G}\rho_{2\beta}(-\log \rho_{\beta})^{(p_*-1)(m(\alpha)-1)}\,dV,
 \end{equation}
 where $\beta=(\beta_1,\dots, \beta_n)$ with $\beta_j=\alpha_{\sigma^{-1}(j)}$ is a permutation of 
 the tuple $\alpha$, the set $G=F_\J\circ S^{-1}$ is
 measurable and contained in the set $ \D^\ell_{\frac{1}{2}} \times \left(\D^{n-\ell}\setminus \D^{n-\ell}_{\frac{1}{2}}\right)$ of \eqref{eq-fl}.
 Note  that $m(\alpha)=m(\beta),$ $\norm{\alpha}_\infty=\norm{\beta}_\infty$ and  consequently, 
$\displaystyle{ p_*=  \frac{2 \norm{\beta}_\infty+2}{\norm{\beta}_\infty+2}}$. The estimate \eqref{eq-G} now follows from conclusion~\eqref{eq-specialcase1} of
Proposition~\ref{prop-specialcase}.

The same change of variables transforms \eqref{eq-FJ2} into the equivalent inequality:
\begin{equation}
 \label{eq-G2}
     \norm{P^+_{\mathbb{D}^n}(\rho_\beta\cdot1_{G})}^{p}_{L^{p}(\mathbb D^n)}\lesssim \int_{G}\rho_{2\beta}\,dV,   
\end{equation}
with the same notation as above. This now follows from 
conclusion \eqref{eq-specialcase2} of Proposition~\ref{prop-specialcase}.

\section{Proofs of Theorem~\ref{thm-genest} and Proposition~\ref{prop-restricted}}\label{sec-pullback}

\subsection{Transformation of restricted type under proper maps}
In order to deduce Theorem~\ref{thm-genest} from the main estimate
Theorem~\ref{thm-polydisc}, we will now consider how restricted  type estimates transform under proper holomorphic maps of domains,
provided the map is of ``quotient type" in the language 
of \cite{CJM}
or a ``ramified normal cover" in the language of \cite{dallara}.  Let  $\phi:\Omega\to \Uu$ be a proper 
holomorphic map  
of  bounded domains in $\cx^n$ which is a 
ramified normal covering, with the finite group of 
deck-transformations $\Gamma$.  Recall that this means
that  each $g\in \Gamma$ is a biholomorphic automorphism of $\Omega$ such that $\phi\circ g =\phi,$ and the restriction $\phi: \Omega\setminus\{\det \phi'=0\}\to \Uu\setminus \phi(\{\det \phi'=0\})$ to the regular points is a  normal (or regular, or Galois) covering, so that the action on $\Gamma$ on each fiber 
$\phi^{-1}(z)$ is transitive. Then $\Uu$ is naturally 
identified with the quotient $\Omega/\Gamma$ in the category of analytic spaces. Without loss of generality, after applying a linear change of coordinates, we can assume that  $\Gamma\subset U(n)$, where $U(n)$
is the  group of unitary linear automorphisms of $\cx^n$.

Let $V$ be a fundamental domain for the action of $\Gamma$ on $\Omega$, i.e., $V$ is an open set that is mapped by $\phi$ biholomorphically into $\Uu$, and $\Uu\setminus\phi(V)$ is a closed set of measure zero. 
Let $\phi_e^{-1}$ be the inverse of $\phi|_V$. Then 
$\{gV:g\in \Gamma\}$ is a partition of $\Omega$ (up to a small set) into fundamental domains, and $\{g\circ\phi_e^{-1}, g\in \Gamma\}$ are the branches of the inverse of $\phi$.

Let $u$ be the nonnegative function on $\Omega$ given by
\begin{equation}
    \label{eq-udef2}
    u=\abs{\det\phi'}.
\end{equation}
 We also define a function $v$ on $\Uu$ by
 \begin{equation}
     \label{eq-vdef}v= \abs{\det \phi' \circ \phi_e^{-1}}={\abs{\det (\phi_e^{-1})'}^{-1}},
 \end{equation}
 where the second representation follows from the
 chain rule. Since $(g\circ\phi_e)'= g \circ \phi_e'$,
 and $\abs{\det g=1}$ as $g:\cx^n\to \cx^n$ is unitary,
 this in fact defines $v$ as a single-valued globally defined function on $\Uu\setminus \phi(\{\det\phi'=0\}) $.  
 \begin{prop} \label{prop-cov}For each measurable subset $E\subset \Uu$ and each $p\geq 1$:
     \[ \norm{P^+_{\Uu}(1_E)}_{L^p(\Uu, v^{p-2})} \leq {\left(\#\Gamma\right)^{-\frac{1}{p}}}\norm{P_\Omega^+(u\cdot 1_{\phi^{-1}(E)})}_{L^p(\Omega)}.\]
 \end{prop}
 
 For a function $f$ on $\Uu$, define a function 
$\phi_\flat f$ on $\Omega$ by
\[ \phi_\flat(f)= (f\circ \phi) \abs{\det \phi'} =u\cdot (f\circ \phi).\]
Let $[L^p(\Omega)]_\Gamma$ denote the closed linear subspace of 
$L^p(\Omega)$ consisting of functions $f$ which are $\Gamma$-invariant in 
the sense that for each $g\in \Gamma$ we have
\[g_\flat f = (f\circ g)\cdot \abs{\det g}= f.\]
Then we have
\begin{lem}\label{lem-flat} For each $p\geq 1$, the map $\phi_\flat$ is an  isomorphism of the Banach space $L^p(\Uu, v^{p-2})$
with the subspace $[L^p(\Omega)]_\Gamma$ of 
$L^p(\Omega)$. 
 In fact the map 
\[ \phi_\flat: L^p(\Uu, v^{p-2})\to [L^p(\Omega)]_\Gamma\]
is an isometry up to a constant factor:
\[ \norm{\phi_\flat f}_{L^p(\Omega)} = (\#{\Gamma})^{\frac{1}{p}} \norm{f}_{L^p(\Uu, v^{p-2})}.\]
\end{lem}
\begin{proof} Let $f\in L^p(\Uu, v^{2-p})$ and
let $h=\phi_\flat f =(f\circ \phi)\cdot \abs{\det \phi'}$. Then for $g\in \Gamma$ we have
$g_\flat h = g_\flat(\phi_\flat f)=(\phi\circ g)_\flat f =\phi_\flat f =h$, using the fact that
$g$ is a deck-transformation of $\phi.$ Therefore
to show that $\phi_\flat$ maps $L^p(\Uu, v^{2-p})$ into $[L^p(\Omega)]_\Gamma$, it suffices to show that $h\in L^p(\Omega)$.

We have
\begin{align}
    \norm{h}^p_{L^p(\Omega)}&=\norm{\phi_\flat f}_{L^p(\Omega)}^p=\int_{\Omega}\abs{f\circ \phi\cdot \det \phi'}^p dV= \sum_{g\in \Gamma} \int_{gV}\abs{f\circ \phi}^p\cdot \abs{ \det \phi'}^p dV  \nonumber\\
    &=\sum_{g\in \Gamma} \int_{\Uu}\abs{f\circ \phi\circ (g \phi_e^{-1})}^p\cdot \abs{\det \phi'\circ (g\phi_e^{-1})}^{p} \cdot \abs{\det(g\phi_e^{-1})'}^2dV\label{eq-biholo}\\
    &= \sum_{g\in \Gamma} \int_{\Uu}\abs{f\circ \phi\circ \phi_e^{-1}}^p\cdot \abs{\det \phi'\circ \phi_e^{-1}}^{p} \cdot \abs{\det(\phi_e^{-1})'}^2dV\nonumber\\
    &= (\#\Gamma)\cdot\int_{\Uu}\abs{f}^p \cdot\abs{\det \phi'\circ \phi_e^{-1}}^{p-2}\,dV,\label{eq-laststep}
     \end{align}
where in \eqref{eq-biholo} we have used the fact that $g\phi_e^{-1}:\Uu\to gV$ is a biholomorphism, and in \eqref{eq-laststep} we have used \eqref{eq-vdef}. A standard argument (see \cite[Proposition~5.14]{advances}) shows 
surjectivity onto $[L^p(\Omega)]_\Gamma$.
\end{proof}
\begin{lem}\label{lem-bell} We have, for each measurable subset $E\subset\Uu$,
\[ u\cdot P_\Uu^+(1_E)\circ \phi \leq  P_\Omega^+(u\cdot 1_{\phi^{-1}(E)}),\]
where both sides of the inequality are nonnegative functions on $\Omega$.
\end{lem}
\begin{proof} Thanks to the Bell transformation formula for the Bergman kernel under a biholomorphic mapping, we have for $z\in \Omega$ and $w\in \Uu$
\[\det \phi'(z) K_{\Uu}(\phi(z),w)= \sum_{g\in \Gamma} K_\Omega(z, g\circ\phi_e^{-1}(w))\cdot \ol{\det (g\circ \phi_e^{-1})'(w)},\]
which leads to 
\begin{equation}\label{eq-absbergman1}
    \abs{\det \phi'(z)}\cdot \abs{K_{\Uu}(\phi(z),w)} \leq  \sum_{g\in \Gamma} \abs{K_\Omega(z, g\circ\phi_e^{-1}(w))}\cdot \abs{\det (g\circ \phi_e^{-1})'(w)}.\end{equation}
   For $z\in \Omega$, we have
    \begin{align}
        (u\cdot P_\Uu^+(1_E)\circ \phi) (z)
        &=\abs{\det \phi'(z)}\cdot P_\Uu^+(1_E)(\phi(z))\nonumber\\
        &= \int_\Uu \abs{\det \phi'(z)}\cdot\abs{K_\Uu(\phi(z),w)} 1_E(w)dV(w)\nonumber\\
        &\leq \sum_{g\in \Gamma}\int_\Uu \abs{K_\Omega(z, g\circ\phi_e^{-1}(w))}\cdot \abs{\det (g\circ \phi_e^{-1})'(w)}1_E(w)dV(w) \label{eq-comp3}
    \end{align}
    where in \eqref{eq-comp3} we have used \eqref{eq-absbergman1}. If $V$ is a fundamental domain for the action of $\Gamma$ on $\Omega$,
    for $g\in \Gamma$, the restricted map $\phi:gV\to \phi(gV)\subset \Uu\setminus Z$ is a biholomorphism onto its image, whose inverse is $g\circ \phi_e^{-1}$. Therefore by the change of variables formula, we have
    \begin{align}
        \eqref{eq-comp3}&=\sum_{g\in \Gamma}\int_{gV}\abs{K_\Omega(z, \zeta)}\cdot \abs{\det (g\circ \phi_e^{-1})'(\phi(\zeta))}\cdot 1_E(\phi(\zeta))\cdot \abs{\phi'(\zeta)}^2dV(\zeta)\nonumber\\
        &= \int_{\Omega}\abs{K_\Omega(z, \zeta)}\cdot \abs{\det (g\circ \phi_e^{-1})'(\phi(\zeta))}\cdot 1_E(\phi(\zeta))\cdot \abs{\phi'(\zeta)}^2dV(\zeta) \label{eq-comp4}\\&=P_\Omega^+(u\cdot 1_{\phi^{-1}(E)})(z).\nonumber
        \end{align}
    where in \eqref{eq-comp4}, we have used the fact (derived from the chain rule applied to $(g\circ \phi_e^{-1})\circ \phi=\mathrm{id}_{gV}$) that
    $\det (g\circ \phi_e^{-1})'(\phi(\zeta))\cdot \det \phi'(\zeta)=1.$
\end{proof}
\begin{proof}[Proof of Proposition~\ref{prop-cov}]
We have 
\begin{align*}
\norm{P^+_{\Uu}(1_E)}_{L^p(\Uu, v^{p-2})}&= (\#\Gamma)^{-\frac{1}{p}} \norm{\phi_\flat P^+_{\Uu}(1_E)}_{L^p(\Omega)}& \text{ by Lemma~\ref{lem-flat}}\\
&= (\#\Gamma)^{-\frac{1}{p}} \norm{u\cdot P^+_{\Uu}(1_E)\circ \phi }_{L^p(\Omega)}\\
&\leq (\#\Gamma)^{-\frac{1}{p}} \norm{ P_\Omega^+(u\cdot 1_{\phi^{-1}(E)}) }_{L^p(\Omega)}.&\text{ by Lemma~\ref{lem-bell}}
\end{align*}
    \end{proof}
\subsection{Facts about monomial mappings and  monomial polyhedra} We now recall some notation and results of \cite{CJM}. We will use the ``vector and matrix power" notation generalizing the classical multi-index notation to simplify the writing.
If $\alpha=(\alpha_1,\dots,\alpha_n)$ is a $1\times n$ row vector, and $z=(z_1,\dots, z_n)^T$ is column vector of  size $n\times 1$, we will 
denote
\begin{equation}
    \label{eq-vectorpower}
    z^\alpha =z_1^{\alpha_1}\dots z_n^{\alpha_n}=\prod_{j=1}^n z_j^{\alpha_j},
\end{equation}
whenever the powers $z_j^{\alpha_j}$ make sense, and where we use the convention 
$0^0=1.$

For an $n\times n$ matrix $P$, we denote the element at the $j$-th row and $k$-th column of $P$ by $p_k^j$ and the $j$-th row of $P$ by $p^j$, which we think of as a multi-index of size $n$.
We then define for a column vector $z$ of size $n$,
\begin{equation}
    \label{eq-matrixpower}
    z^P=\begin{pmatrix}
     z^{p^1}\\\vdots\\ z^{p^n}
    \end{pmatrix}= \begin{pmatrix}
     z_1^{p^1_1}z_2^{p^1_2}\cdots z_n^{p^1_n}\\\vdots\\  z_1^{p^n_1}z_2^{p^n_2}\cdots z_n^{p^n_n}
    \end{pmatrix}.
\end{equation}
The following computation generalizes the formula $\frac{d}{dx}x^\alpha= \alpha x^{\alpha-1}$  (see \cite[Lemma~3.8]{CJM} or \cite[Lemma~4.2]{nagelpramanik}):
\begin{prop}\label{prop-derivative} Let $\phi(z)=z^A$. Then $ \det \phi'(z)=\det A\cdot z^{\one A-\one}.$
    \end{prop}
By permuting the rows of the matrix $B$ defining $\Uu_B$
we can show the following:
\begin{prop}[{\hspace{1sp}\cite[Proposition~3.2]{CJM}}]
    \label{prop-mono1}
The matrix $B$ can be assumed to satisfy the conditions
       \begin{enumerate}
           \item $\det B>0$, and
           \item each entry of the inverse $B^{-1}$ is nonnegative.
       \end{enumerate}
\end{prop}
We next recall the following  description of a monomial polyhedron as a quotient domain.  By a \emph{punctured polydisc} we mean a domain $\Omega$ in $\cx^n$ which can be represented as a product of one dimensional domains
$ \Omega= \Omega_1\times \dots\times \Omega_n=\{z\in \cx^n: z_j\in \Omega_j\}$,
where each $\Omega_j$ is either a disc $\{\abs{z_j}<1\}\subset \cx$ or a punctured disc $\{0<\abs{z_j}<1\}\subset \cx.$

With this notation, we can state the following result, a version of which was proved in \cite{CJM}:
\begin{prop} \label{prop-mono2}
   Let  $\Uu_B$ be a monomial polyhedron as in 
   \eqref{eq-udef}, where $B$ satisfies the conditions
   of Proposition~\ref{prop-mono1},
and let $A$
        be as in \eqref{eq-adef}. Then there is a punctured polydisc $\Omega$ such that the  map
        $\phi:\Omega\to \Uu_B$ given by
        $\phi(z)=z^A$ defines a  proper holomorphic map of quotient type of degree $\abs{\det A}$. The deck-transformation group $\Gamma$ consists of the automorphisms of $\Omega$ of the form
\begin{equation}
    \label{eq-sigmanu}
    \sigma_\nu(z)=\left(e^{2\pi i c^1\nu}z_1,e^{2\pi ic^2 \nu}z_2, \dots, e^{2\pi i c^n \nu}z_n\right), \quad
    \nu\in \Z^n,
\end{equation}
        where $c^1,\dots, c^n\in \mathbb{Q}^{1\times n}$ are the rows of the matrix $C=A^{-1}.$
        
\end{prop}
\emph{Remarks on the proof of Proposition~\ref{prop-mono2}:} This is essentially Theorem~\ref{thm-polydisc}2 of 
\cite{CJM}, with the exception that there
the matrix $A$  of \eqref{eq-adef} is replaced by the adjugate $\Delta=\adj B =(\det B)\cdot B^{-1}$ of 
$B$,  where  $A$ is obtained from $\Delta=\adj B$ by dividing each column by its gcd. 
We can write 
$ \Delta=D\cdot A,$
where $A$ is as in \eqref{eq-adef}, and 
$D$ is the diagonal matrix 
$D=\mathrm{diag}(\gcd(\delta_1),\dots, \gcd(\delta_n))\in\Z^{n\times n}.$
Since $z\mapsto z^D$ is a proper holomorphic map
of $\Omega$ onto itself, and $z^{DA}=(z^A)^D$ it follows that the maps $z\mapsto z^\Delta$ and $z\mapsto z^A$ are both
proper holomorphic maps of quotient type from $\Omega$ into $\Uu_B$. The proof of
\cite[Theorem~\ref{thm-polydisc}2]{CJM} only uses this fact. 

\subsection{Proofs of Theorem~\ref{thm-genest} and Proposition~\ref{prop-restricted}} 
The following is an immediate consequence of Proposition~\ref{prop-cov} and the description of $\Uu_B$ as a quotient domain given by Proposition~\ref{prop-mono2}:
 \begin{prop}
     \label{prop-cov2}
     Let $\Uu_B\subset\cx^n$ be as in \eqref{eq-udef},
     let $A$ be the matrix in \eqref{eq-adef}, and let $w=  \rho_{\one - \one A^{-1}}$ be as in \eqref{eq-vdef-poly}. 
     Then for $p> 1$ we have for each measurable subset $E\subset \Uu_B$:
\begin{equation}
     \norm{P^+_{\Uu_B}(1_E)}_{L^{p}(\Uu_B,w^{p-2})} 
    \lesssim \norm{P^+_{\D^n}(\rho_{\one A-\one}\cdot 1_{\phi^{-1}(E)})}_{L^{p}(\D^n)}.
\end{equation}
\end{prop}
\begin{proof}
We apply Proposition~\ref{prop-cov} to the branched covering map  $\phi:\Omega\to \Uu_B$  given by $\phi(z)=z^A,$
of Proposition~\ref{prop-mono2}. Let $C=A^{-1},$
which has nonnegative rational entries by  Proposition~\ref{prop-mono1} above. 
Then a branch of the inverse of $\phi$ is given by $\phi_e^{-1}=z^C$, where the powers $z^{c^j_j}$ needed in definition \eqref{eq-matrixpower} of matrix powers are taken to be arbitrary branches of the rational powers. Further, we have
$\det \phi'(z)= \det A \cdot z^{\one A -\one}$ by 
Proposition~\ref{prop-derivative},
so the function $u=\abs{\det \phi'}$ of \eqref{eq-udef2} that occurs in Proposition~\ref{prop-cov} is given in this case by
\[u(z)= \abs{\det \phi'(z)}= \abs{\det A}\cdot\abs{z^{\one A -\one}}=\abs{\det A}\cdot \rho_{\one A-\one}(z).\]
For this map $\phi$, the function $v$ of \eqref{eq-vdef} specializes to
\begin{align}
    v(\zeta)&=\abs{\det \phi'\circ \phi_{e}^{-1}(\zeta)}=\abs{\det A} \abs{(\zeta^{A^{-1}})^{\one A-\one}}\nonumber\\
    &=\abs{\det A} \abs{\zeta^{(\one A-\one)A^{-1}}}
    = \abs{\det A} \abs{\zeta^{\one -\one A^{-1}}}\nonumber=\abs{\det A} \rho_{\one-\one A^{-1}}(\zeta)\nonumber\\
    &= \abs{\det A} w(\zeta).\label{eq-vw}
\end{align}
Since $v$ and $w$ differ only in a constant factor,
\[ \norm{P^+_{\Uu_B}(1_E)}_{L^p(\Uu_B, v^{p-2})}\approx \norm{P^+_{\Uu_B}(1_E)}_{L^p(\Uu_B, w^{p-2})},\]
and since $\Omega$ differs from $\D^n$ only in a set of measure zero,  
\[ \norm{P_\Omega^+(u\cdot 1_{\phi^{-1}(E)})}_{L^p(\Omega)}= \norm{P^+_{\D^n}(\abs{\det A}\cdot \rho_{\one A-\one}\cdot 1_{\phi^{-1}(E)}  )}_{L^p(\D^n)} \approx \norm{P^+_{\D^n}( \rho_{\one A-\one}\cdot 1_{\phi^{-1}(E)}  )}_{L^p(\D^n)}, \]
the result follows by Proposition~\ref{prop-cov}. 
\end{proof} 
 
\begin{proof}[Proof of Theorem~\ref{thm-genest}]
    By Proposition~\ref{prop-cov2}, for each
measurable $E\subset \Uu_B$, we have
\begin{align}
    \norm{P^+_{\Uu_B}(1_E)}_{L^{p_*}(\Uu_B,w^{p_*-2})}^{p_*} &\lesssim \norm{P^+_{\D^n}(\rho_{\one A-\one}\cdot 1_{\phi^{-1}(E)})}_{L^{p_*}(\D^n)}^{p_*}\nonumber\\
    & \lesssim  \int_{\phi^{-1}(E)} \rho_{2(\one A-\one)}\cdot(-\log \rho_{\one A-\one})^{(p_*-1)(m(\one A-\one)-1)}\,dV, \label{eq-comp11} \end{align}
 using inequality \eqref{eq-poyest1} of Theorem~\ref{thm-polydisc}. Now from the definition~\eqref{eq-malpha},
 we see that $m(\one A-\one)$ is the number of largest elements in the $n-$tuple $\one A-\one=(\one a_1-1,\one a_2-1,\dots, \one a_n-1)$ and therefore coincides with the quantity $m$ of \eqref{eq-m} associated to the 
 monomial polyhedron $\Uu_B$. Therefore, since $\abs{\det \phi'}\approx\rho_{\one A-\one}$ by Proposition~\ref{prop-derivative}:
  
    \begin{align}
  \eqref{eq-comp11} &\approx \int_{\phi^{-1}(E)} (-\log \abs{\det \phi'})^{(p_*-1)(m-1)}\,\abs{\det \phi'}^2dV\nonumber\\
    &\approx \int_E (-\log \abs{\det \phi'\circ \phi_{e}^{-1}})^{(p_*-1)(m-1)}\,dV\label{eq-detatoone}\\
    &=\int_E (-\log v)^{(p_*-1)(m-1)}\,dV
    &\text{using the definition \eqref{eq-vdef}}\nonumber\\
        &\approx\int_E (-\log w)^{(p_*-1)(m-1)}\,dV
&\text{ $v\approx w$ from \eqref{eq-vw}}\nonumber\\
    &=\norm{1_E}_{L^{p_*}(\Uu_B, (-\log w)^{(p_*-1)(m-1)})}^{p_*},\nonumber
\end{align}
where in \eqref{eq-detatoone}, we have used  change of variables with the $\abs{\det A}$-to-one map $\phi$. This completes the proof.
\end{proof}

\begin{proof}[Proof of Proposition~\ref{prop-restricted}]
    This follows the same ideas as the proof of Theorem~\ref{thm-genest} above. Let $p_*<p\leq 2$. Then we have
    \begin{align*}
        \norm{P^+_{\Uu_B}(1_E)}_{L^p(\Uu_B)}^p& \leq
        \norm{P^+_{\Uu_B}(1_E)}_{L^p(\Uu_B, w^{p-2})}^p
        & \text{since $w^{p-2}\geq 1$ on $\Uu_B$}\\
        & \lesssim \norm{P^+_{\D^n}(\rho_{\one A -\one} 1_{\phi^{-1}(E)})}_{L^p(\D^n)}^p
        &\text{by Proposition~\ref{prop-cov2}}\\
         & \lesssim  \int_{\phi^{-1}(E)} \rho_{2(\one A-\one)}\,dV
    &\text{ using inequality \eqref{eq-polyest2} of Theorem~\ref{thm-polydisc}}\\
    &\approx \int_{\phi^{-1}(E)} \abs{\det \phi'}^2dV\\
        &\approx \int_E dV = \abs{E}= \norm{1_E}_{L^p(\Uu_B)}^p,
    \end{align*}
    where the justification of the last steps is the same as in the proof of Theorem~\ref{thm-genest} above. This completes the proof.
\end{proof}
 \section{Proof of Theorem~\ref{thm-irreg1}}\label{sec-irreg}
For an integer multi-index $\beta\in \Z^{1\times n}$, denote by 
$e_\beta$ the  monomial function $e_\beta(z)=z^\beta$. Set $\alpha=\one A-\one,$ and 
let $U=\det \phi',$ where $\phi:\Omega\to \Uu_B$ is the branched covering map of Proposition~\ref{prop-mono2}.
Thanks to  Proposition~\ref{prop-derivative}, 
\[U(z)= \det\phi'(z)= \det A \cdot e_{\alpha}(z)= 
\det A \cdot e_{\one A-\one}(z)= \det A\cdot z^\alpha.
\]
After permuting the coordinates, without loss of generality, we assume that $$\alpha_1=\alpha_2=\dots=\alpha_m=\max_j\{\alpha_j\}.$$ 

Recall $\Gamma$ is the deck-transformation group \eqref{eq-sigmanu} of Proposition~\ref{prop-mono2}. A function $f$ on $\Omega$ is called  \emph{$\Gamma$-invariant} if
$f\circ g =f$ for each $g\in \Gamma$, where
each of
the automorphisms of $\Omega$ in \eqref{eq-sigmanu} is a Reinhardt rotation, it follows that each function $z\mapsto \abs{z_j}, 1\leq j\leq n$ is $\Gamma$-invariant. 

A direct computation shows that with $\sigma_\nu$ as in \eqref{eq-sigmanu}, we have $e_\beta\circ \sigma_\nu=e^{2\pi i \beta A^{-1}\nu}e_\beta,$ for $\nu\in \Z^{n\times 1}$, where $e_\beta(z)=z^\beta$ is a monomial. Therefore, $e_\beta$ is $\Gamma$-invariant if and only if there is an integer 
row vector $\mu\in \Z^{1\times n}$ such that $\beta=\mu A.$ Theorem~\ref{thm-irreg1} will be a consequence of the following:

\begin{thm}\label{thm-irreg2} With notation and conventions as above,
    let $\Uu_B$ be a nontrivial monomial polyhedron
with $m>1$ and $\alpha_1=\alpha_2=\dots=\alpha_m=\max_j\{\alpha_j\}.$ Suppose there exists a $\Gamma$-invariant monomial $e_b$ with $b_j=b_k=1$ for some $j<k\leq m$.  Then the Bergman projection $P_{\Uu_B}$ is not of restricted  weak type $(p_*,p_*)$.
\end{thm}
We verify that there are domains that satisfy 
the hypotheses of this theorem.
\begin{prop}\label{prop-example}  The domain $\Uu_B\subset\cx^3$
of \eqref{eq-Bpoly}  is a monomial polyhedron, for which  $m=2$, and the monomial $e_b(z)=z_1z_2z_3$
corresponding to $b=(1,1,1)$
is $\Gamma$-invariant, for which 
$b_1=b_2=1$. Furthermore, $\Uu_B$ is biholomorphic to a punctured polydisc in $\cx^3$.
    \end{prop}
    Therefore the hypotheses of Theorem~\ref{thm-irreg2} are verified for the domain $\Uu_B$ of \eqref{eq-Bpoly}.  
\begin{proof}
It is routine to verify that $\Uu_B$ as in \eqref{eq-Bpoly} is the monomial polyhedron defined by the matrix
\begin{equation*}
    \label{eq-Bmatrix}
B=\begin{pmatrix}
    \phantom{-}1&\phantom{-}0&0\\-1&\phantom{-}1&0\\\phantom{-}1&-1&1
\end{pmatrix}.   
\end{equation*}

Since $\det B=1$, and each row of $B$ has gcd 1, the adjugate
$\Delta=\adj B$ and the matrix $A$ of \eqref{eq-adef} are both
given by
\[ \Delta=A=\begin{pmatrix}
    1&0&0\\1&1&0\\0&1&1
\end{pmatrix}, \]
Since $\det A=1$, it is clear that $\Uu_B$ is biholomoprhic to a punctured polydisc (i.e. a product of discs and punctured discs, see \cite{CJM}).
Then, 
\[ \one A =(2,2,1), \quad (1,0,1)A =(1,1,1). \]
By \eqref{eq-m} we have $m=2$ and by the comments preceding the statement of Theorem~\ref{thm-irreg2}, the monomial $e_b$ with  $b=(1,1,1)$ is $\Gamma$-invariant. 
\end{proof}

\subsection{Preliminaries}
In terms of the quasinorm  of \eqref{eq-quasi}, we want to show that \[\displaystyle{\sup_E \frac{\norm{P_{\Uu_B}1_E}^*_{L^{p_*,\infty}(\Uu_B)}}{\norm{1_E}_{L^{p_*}(\Uu_B)}}=\infty},\]
where the supremum is over measurable subsets $E$ of $\Uu_B$ whose Lebesgue measure is positive. It suffices to find for each small $s>0$ a measurable subset $E_s\subset \Uu_B$ and a $\lambda_s>0$ such that 
\begin{equation}
    \label{eq-notweak}
    \lim_{s\to 0} \frac{(\lambda_s)^{p_*} \cdot\abs{\left\{w\in \Uu_B:\abs{P_{\Uu_B}1_{E_s}(w)}>\lambda_s\right\}}}{\abs{E_s}}=\infty.
\end{equation}
Thanks to Bell's transformation formula for the Bergman projection under a proper map, recalling that $U=\det\phi'$, we have for $E_s\subset \Uu_B$:
\begin{align*}
    (P_{\Uu_B}(1_{E_s})\circ \phi)\cdot U= P_\Omega(1_{E_s}\circ \phi\cdot \det \phi')
    =P_\Omega(\det\phi'\cdot 1_{\wt{E}_s})=P_{\D^n}(U\cdot 1_{\wt{E}_s})
\end{align*}
where we have set $\wt{E}_s=\phi^{-1}(E_s)\subset \Omega,$ so that  $1_{E_s}\circ \phi=1_{\wt{E}_s}$, and
we have identified the Bergman projection of $\Omega$ with that of the unit polydisc $\D^n.$
Notice that the set $\wt{E}_s$ is invariant under the group $\Gamma,$ i.e. 
for each $g\in \Gamma$, we have $g(\wt{E}_s)=\wt{E}_s$.
Further, if $\wt{E}_s$ is any $\Gamma$-invariant set in $\Omega$,
then the above holds with $E_s=\phi(\wt{E}_s).$

 Then for $\lambda>0$
\begin{align}
    \abs{\{w\in \Uu_{B}:\abs{P_{\Uu_{B}}(1_{E_s})(w)}>\lambda\}}
    =&\mu\{z\in \Omega: \abs{P_{\Uu_{B}}(1_{E_s})(\phi(z))}>\lambda\}\nonumber\\
    =&\mu\{z\in \D^n: {\abs{P_{\D^n}(u\cdot 1_{\tilde E_s})(z)}}/{\abs{u}}>\lambda\}\nonumber\\
    =& \mu\left\{z\in \D^n: \frac{\abs{P_{\D^n}(u\cdot 1_{\tilde E_s})(z)}}{\abs{\det A}\abs{z^\alpha}}>\lambda\right\}.\label{eq-disttransf}
\end{align}
Moreover, the  denominator
of \eqref{eq-notweak} is transformed to
\begin{align}
\abs{E_s}=\mu(\wt{E}_s)\approx \int_{\wt{E}_s}  \rho_{2\alpha}\,dV.
\label{eq-denotrans}
\end{align}
Now let\begin{equation}
    \label{eq-hsdef}
    h_s= P_{\D^n}(U\cdot 1_{F_s})\in A^2(\D^n).
\end{equation}
Using \eqref{eq-disttransf} and \eqref{eq-denotrans} in \eqref{eq-notweak} we see that  to prove
Theorem~\ref{thm-irreg1}, we need to find
for each small $s>0$, a $\Gamma$-invariant measurable subset $F_s\subset \Omega$ 
and a $\lambda_s>0$ such that 
\begin{equation}
    \label{eq-notweak2}
     \lim_{s\to 0} \frac{(\lambda_s)^{p_*} \cdot\mu\left\{ \frac{\abs{ h_s}}{\rho_\alpha}>\lambda_s\right\}}{\int\limits_{F_s}\rho_{2\alpha}\,dV}=\infty.
\end{equation}
\subsection{The sets $F_s$}
Without loss of generality, we assume that the $\Gamma$-invariant monomial $e_b$ has 
$b_1=b_2=1$. By perhaps multiplying $e_b$ by $z_j^{|\det A|}$ for $j>2$, we may further assume that $b\succeq \one$, i.e. $b_j\geq 1$ for all $j$.

 Let $B\subset [0,2\pi)^n$
be the subset 
\begin{align*}
   B&=\left\{\theta\in  \rl^{n\times 1}: \theta\in [0,2\pi)^n, \sin( (\alpha-b+\one)\theta)\geq 0\right\}\\
   &=
\bigcup_{k\in \Z}\left\{\theta\in[0,2\pi)^n:2k\pi\leq\sum{(\alpha_j+1-b_j)\theta_j}\leq 2k\pi+\pi\right\}.
\end{align*}
For each $s\in(0,1/4)$,
define a subset $A(s)\subset [0,1)^n$ by
\[ A(s)=\{r\in [0,1)^n:{r_{1}}r_{2}<s, s<{r_{1}}<\sqrt{s}\}.\]
We now set
\begin{align}
    F_s&=\left\{(r_1e^{i\theta_1},\dots, r_n e^{i\theta_n})\in\cx^n: r\in A(s), \theta\in B\right\}
    \nonumber\\
   &= \left\{z\in\D^n: \abs{z_{1}z_2}<s,s< \abs{z_{1}}<\sqrt{s}, 0\leq\text{Arg}(z^{\alpha+\one-b})\leq {\pi}\right\}
   \label{eq-Fsdef}.
\end{align}
 Since each of the functions $z\mapsto\abs{z_1}, z\mapsto \abs{z_2}$ 
and $e_{\alpha+\one-b}=e_{\one A -b}$ is $\Gamma$-invariant, it follows 
that for each $g\in \Gamma$, we have $g(F_s)=F_s$, i.e. $F_s$ is invariant under the group $\Gamma$. 

We turn to computing  the integrals of monomials  over $A(s)$. For $d\in\mathbb N^n$,
\begin{align}
\int_{A(s)}r^{d}dV(r)\nonumber=&\frac{1}{\prod_{j=3}^n(d_j+1)}
\int_{\{r\in[0,1]^2:r_1r_2<s,s<r_1<\sqrt s\}}r_1^{d_1}r_2^{d_2}dr_1dr_2\nonumber\\=&\frac{1}{\prod_{j=3}^n(d_j+1)}\int_{s}^{\sqrt s}\int_{0}^{s/r_1}r_1^{d_1}r_2^{d_2}dr_2dr_1\nonumber\\=&\frac{1}{\prod_{j=3}^n(d_j+1)}\int_{s}^{\sqrt s}r_1^{d_1}\frac{s^{d_2+1}r_1^{-d_2-1}}{d_2+1}dr_1\nonumber\\=&\label{10.2}\begin{cases}
  \dfrac{1}{(d_1-d_2)\prod_{j=2}^n(d_j+1)}\left(s^{\frac{d_1+d_2}{2}+1}-s^{d_1+1}\right)& \text{ for } d_1\neq d_2,\\
    \dfrac{1}{2\prod_{j=2}^n(d_j+1)}s^{d_1+1}\log(1/s)& \text{ for } d_1=d_2.
\end{cases}
\end{align}
Using this  we obtain for the denominator of \eqref{eq-notweak2}
\begin{align}
   \int_{F_s}\rho_{2\alpha}\,dV \leq &\int_{\{z\in\mathbb D^n:\abs{z_{1}}\abs{z_{2}}<s,s<\abs{z_{1}}<\sqrt s\}}\rho_{2\alpha}\,dV\nonumber\\
   \approx&\int_{\{z\in\mathbb D^2:\abs{z_{1}}\abs{z_{2}}<s,s<\abs{z_{1}}<\sqrt s\}}|z_1z_2|^{2\alpha_1}\,dV\nonumber\\
   \approx &\int_{\{r\in[0,1]^2:r_1r_2<s,s<r_1<\sqrt s\}}
   r_1^{2\alpha_1+ 1}r_2^{2\alpha_1+ 1}
   dr_1dr_2\nonumber\\
   \approx&s^{(2\alpha_1+2)}\log (1/s).\label{eq-denoest}
 \end{align}  
 \subsection{A lower bound on  $h_s$} In the definition \eqref{eq-hsdef} of $h_s$ we let the set $F_s$ to be as in \eqref{eq-Fsdef}. We will show the following:
\begin{lem}\label{lem-hs} There are constants $r\in (0,\frac{1}{4})$,
$s_0\in (0,\frac{1}{4})$ and $K>0$ and a relatively compact nonempty open subset $D\subset \D^{n-2}$ 
such that  for $z$ in 
the set
\begin{align*}
    \displaystyle{\Pi=\mathbb D^2_r\times D=\left\{z\in \cx^n: \abs{z_j}<r \text{ for } 1\leq j \leq 2, (z_3,\dots, z_n)\in D \right\}}
\end{align*}
and $s\in (0,s_0)$ we have
$ \abs{h_s(z)} \geq K s^{\alpha_1+2} \log\left(1/s\right).$
\end{lem}
\begin{proof}

By Bell's formula we have
\begin{equation}
    \label{eq-hsbell}
    h_s=U\cdot P_{\Uu_B}(1_{\phi(F_s)})\circ \phi= \det A \cdot e_{\one A-\one} P_{\Uu_B}(1_{\phi(F_s)})\circ \phi.
\end{equation}
 The Taylor series expansion
$ h_s =\sum_{\gamma\succeq 0} a_\gamma(s) e_\gamma$ in $\D^n$ is also an orthogonal expansion, and its coefficients are given by
\begin{align}
a_\gamma(s)=&\frac{\ipr{h_s,e_\gamma}}{\norm{e_\gamma}^2_{L^2(\D^n)}}= \frac{\ipr{P_\Omega(u1_{F_s}),e_\gamma}}{\norm{e_\gamma}^2_{L^2(\D^n)}}=\frac{\ipr{u1_{F_s},e_\gamma}}{\norm{e_\gamma}^2_{L^2(\D^n)}}=\frac{(\gamma+\one)^{\one}}{\pi^n}\ipr{u1_{F_s},e_\gamma}\nonumber\\
&=\frac{(\gamma+\one)^\one\cdot \det A}{\pi^n} \int_{F_s}e_\alpha\ol{e_\gamma}dV\nonumber\\
&= \frac{(\gamma+\one)^\one\cdot \det A}{\pi^n}\cdot \int_{A(s)}r^{\gamma+\alpha+\one}dV(r)\cdot \int_B e^{i(\alpha-\gamma)\cdot\theta}dV(\theta). \label{eq-coefficient}
\end{align}

From \eqref{eq-hsbell}, the function $h_s\cdot e_\one$ is $\Gamma$-invariant, since it is the product of $\det A$ and the $\Gamma$-invariant functions $e_{\one A}$ and  $P_{\Uu_B}(1_{\phi(F_s)})\circ \phi$.   Consequently the series
$\sum_{\gamma\succeq 0} a_\gamma(s) e_{\gamma+\one}$ represents the $\Gamma$-invariant function $h_s\cdot e_\one$, and  each term of this series must be $\Gamma$-invariant. Therefore the series for $h_s$ can be rewritten as
\begin{equation}
    \label{eq-hs}
    h_s = \sum_{\substack{\beta \succeq 1\\ \beta \text{ is } \Gamma-\text{invariant}}} a_{\beta-\one}(s) e_{\beta-\one}.
\end{equation}
Let $\mathcal{J}$ be
the subset $\mathcal{J}=\{\beta\succeq \one: \beta \text{ is } \Gamma-\text{invariant}, \beta_1=\beta_2=1\},$
and let $\mathcal{K}$ be the set of
indices of the remaining terms in the series \eqref{eq-hs}, i.e.
$ \mathcal{K}=\{\beta\succeq \one:\beta \text{ is } \Gamma-\text{invariant}, \beta_1\not=1 \text{ or } \beta_2\not=1\}.$
Notice that $\mathcal{J}$ is nonempty since $b\in \mathcal{J}.$
Write 
\begin{equation}
    \label{eq-hsjk}
    h_s=h_s^{\J}+ h_s^{\mathcal{K}}= \sum_{\beta\in \mathcal{J}} a_{\beta-\one}(s) e_{\beta-\one}+\sum_{\beta\in \mathcal{K}} a_{\beta-\one}(s) e_{\beta-\one}.
\end{equation}
The first term can be written as
$\displaystyle{ h_s^{\J}(z)=\sum_{\beta\in \mathcal{J}} a_{\beta-\one}(s) e_{\beta-\one}(z) = a_{b-1}(s)f(z_3,\dots, z_n), }$
where $f$ is a holomorphic function on 
$\D^{n-2}$. We claim $f$ is non-constant and independent of $s$.  To see it's not a constant we only need to check that $a_{b-\one}(s)\not=0$, which follows since in the
formula \eqref{eq-coefficient}, each of the factors is nonzero if $\gamma=b-\one$. For the last factor this follows from 
\[ \mathrm{Im}\left(\int_B e^{i(\alpha-b+\one)\cdot\theta}dV(\theta) \right)= \int_B \sin\left((\alpha-b+\one)\cdot\theta\right)dV(\theta)>0\]
from the definition of $B$. 
Note that by letting $d=\gamma+\alpha+\one$ in (\ref{10.2}) we have
\begin{align}
&\int_{A(s)}r^{\gamma+\alpha+\one}dV(r)=\label{10.3}\begin{cases}
  \frac{1}{(\gamma_1-\gamma_2)\prod_{j=2}^n(\alpha_j+2+\gamma_j)}\left(s^{\frac{\gamma_1+\gamma_2}{2}+\alpha_1+2}-s^{\gamma_1+\alpha_1+2}\right)& \text{ for } \gamma_1\neq \gamma_2,\\
    \frac{1}{2\prod_{j=2}^n(\alpha_j+2+\gamma_j)}s^{\gamma_1+\alpha_1+2}\log(1/s)& \text{ for } \gamma_1=\gamma_2.
\end{cases}
\end{align}
Taking $\gamma=\beta-\one$ in \eqref{eq-coefficient} for $\beta\in \mathcal J=\{\beta\succeq \one: \beta \text{ is } \Gamma-\text{invariant}, \beta_1=\beta_2=1\}$ we
have 
\begin{equation}
    \label{eq-aboneapprox} a_{\beta-\one}(s)= \frac{(\beta)^\one\cdot \det A\cdot s^{\alpha_1+2}\log(1/s)}{2\pi^n(\alpha_1+2)\prod_{j=3}^n(\alpha_j+1+\beta_j)}\int_B e^{i(\alpha-\beta+\one)\cdot\theta}dV(\theta).
\end{equation}
Therefore, 
\[f(z_3,\dots,z_n)=\sum_{\beta\in\mathcal J}\frac{a_{\beta-\one}(s)e_{\beta-\one}(z)}{a_{b-\one}(s)},\] which is independent of $s$. Then there exists a point $a\in \D^{n-2}$ such that $|f(a)|=c>0$. We choose small $r>0$ such that  the set 
\[D=\left\{z\in\mathbb D^{n-2}:\abs{z_j-a_j}< r, 1\leq j\leq n-2\right\}\]
is relatively compact in $\mathbb D^{n-2}$, and $|f(z)|>c/2$ for all $z\in D$. The constant $c$ here again is independent of $s$.

 If   $\beta\in \mathcal{K}$ then either 
$\beta_1>1$,  or $\beta_1=1,\beta_2>1$. If $\beta_1>1$, then for $\gamma=\beta-\one$, $\gamma_1>0$ and  \eqref{10.3} implies $ \displaystyle{\int_{A(s)}r^{\beta+\alpha}dV(r)\lesssim s^{\alpha_1+2}}$ for small $s$. 
If $\beta_1=1$ and $\beta_2>1$,   then for $\gamma=\beta-\one$ we have $\gamma_1\neq \gamma_2$ and hence
$\int_{A(s)}r^{\beta+\alpha}dV(r)\approx s^{\alpha_1+2}$. We then conclude from \eqref{eq-coefficient} that $\abs{a_{\beta-\one}(s)}\lesssim \beta^\one\cdot s^{\alpha_1+2}$ if $\beta\in \mathcal{K}$. Consequently, we can write the summand $h_s^\mathcal{K}$ of \eqref{eq-hsjk} as 
$h_s^\mathcal{K}(z)= s^{\alpha_1+2} g_s(z) $
where $g_s$ is holomorphic in  $\D^n$ and satisfies
\[ \abs{g_s(z)} \lesssim \sum_{\gamma\succeq 0} (\gamma+\one)^\one \abs{z^\gamma} = \dfrac{\abs{z^\one}}{\prod_{j=1}^n (1-\abs{z_j})^2},\]
with the implied constant independent of $s$.

 Since $D$ is relatively compact in $\D^n$
 we have  $|f(z)|>c/2$ for $z\in D$.
 Further, 
$\displaystyle{ \Pi=\D^2_r\times  D,}$
is relatively compact in $\D^n$ and   $\sup_\Pi|g_s|\lesssim 1$. Therefore for $z\in \Pi$
and small $s$
\begin{align}
    \abs{h_s(z)} \geq \abs{h_s^{\mathcal{J}}(z)}- \abs{h_s^{\mathcal{K}}(z)}\nonumber &\geq \abs{a_{b-1}f(z_3,\dots, z_n)}- s^{\alpha_1+2} \abs{g_s(z)}\nonumber\\
    &\geq \frac{cC_1}{2} s^{\alpha_1+2} \log(1/s)  -C_2 s^{\alpha_1+2} \label{eq-hsbound1}
\end{align}
where $C_1, C_2$ are nonnegative constants independent of $s$. In the first term we have used \eqref{eq-aboneapprox} along with the fact that $\abs{f(z_3,\dots, z_n)}>\frac{c}{2}$ if $z\in \Pi$.
Therefore
\[ \abs{h_s(z)}\geq \eqref{eq-hsbound1}=  C_1 s^{\alpha_1+2} \log(1/s)\left(\frac{c}{2} -\frac{C_2}{C_1 \log(1/s)} \right). \]
Since $\frac{C_2}{C_1 \log(1/s)}\to 0$ as $s\to 0$, we can choose $s_0$ so small that 
$ \frac{C_2}{C_1 \log(1/s)}<\frac{c}{2}   $ if $0<s<s_0.$ Thus for $z\in \Pi$ and  $s\in (0,s_0)$ we have $\frac{c}{2} -\frac{C_2}{C_1 \log(1/s)}> \frac{c}{2} $,
and the result follows with $K=\frac{cC_1}{2}.$\end{proof}

\subsection{The numbers $\lambda_s>0$ and completion of the proof} Now with notation as in Lemma~\ref{lem-hs}, we set $\displaystyle{\lambda_s=\left(\frac{1}{4}\right)^{b_2-1} K s^{\alpha_1+2} \log({1}/{s}).}$
Then for $z\in \Pi$ we have
\[ \frac{\abs{h_s(z)}}{\rho_\alpha(z)}\geq \frac{Ks^{\alpha_1+2}\log({1}/{s})}{\abs{z^\alpha}}>\lambda_s,\]
since $\abs{z^\alpha}<1$. 
Consequently, for each $s\in (0,s_0)$,
\[ \mu\left\{z\in \D^n:\frac{\abs{h_s(z)}}{\rho_\alpha(z)}>\lambda_s\right\} \geq \mu(\Pi). \]
Combining the definition of $\lambda_s$ with this statement and \eqref{eq-denoest} we see that:
\[  \frac{(\lambda_s)^{p_*} \cdot\mu\left\{ \frac{\abs{ h_s}}{\rho_\alpha}>\lambda_s\right\}}{\int_{F_s}\rho_{2\alpha}\,dV}\gtrsim \left(\log\left(\frac{1}{s}\right)\right)^{p_*-1}, \]
and therefore goes to $\infty$ as $s\to0.$
Therefore \eqref{eq-notweak2} holds, completing the proof of Theorem~\ref{thm-irreg2}. Combining with Proposition~\ref{prop-example}, Theorem~\ref{thm-irreg1} follows.
\section{Remarks and directions for future research}
Despite  Theorems \ref{thm-genest} and \ref{thm-irreg2}, the connection between the matrix $B$ of $\Uu_B$ and endpoint behavior of $P_{\Uu_B}$ requires further study. Below we list some remarks and questions we would like to investigate in the future:
\begin{enumerate}[wide]
    \item 
In Section 5.4, by realizing that $\int_{F_s}\rho_{2\alpha}\log (1/\rho_{2\alpha})^{t}\,dV\lesssim s^{2\alpha_1+2}(\log(1/s))^{t+1}$, we see that for all $t<p_*-1$,
\[  \frac{(\lambda_s)^{p_*} \cdot\mu\left\{ \frac{\abs{ h_s}}{\rho_\alpha}>\lambda_s\right\}}{\int_{F_s}\rho_{2\alpha}(\log(1/\rho_{2\alpha}))^t\,dV}\gtrsim \left(\log\left(\frac{1}{s}\right)\right)^{p_*-1-t}, \]
which still goes to $\infty$ as $s\to 0$. This leads to the failure of the estimate below for $\Uu_B$ in Theorem \ref{thm-irreg2} and any $t<p_*-1$
\begin{equation}\label{eq-f}
\norm{P^+_{\Uu_{B}}(1_E)}_{L^{p_*}(\Uu_B,w^{p_*-2})}\leq C\norm{1_E}_{L^{p_*}(\Uu_B, (-\log w)^{t})},
    \end{equation}
    thereby implying the sharpness of the exponent $(m-1)(p_*-1)$ of the weight $-\log w$ in Estimate \eqref{eq-genest} for $m=2$. For general $m$, the exponent $(m-1)(p_*-1)$ in \eqref{eq-genest} is likely to be sharp in the sense that there is a monomial polyhedron $\Uu_B$ with \eqref{eq-f} fail for $t<(m-1)(p_*-1)$. 
\item Theorem \ref{thm-irreg2} does not suggest in general the failure of restricted weak type $(p_*,p_*)$ estimates for $P_{\Uu_B}$ when $m>1$. For $\Uu_B\subset\cx^2$, it is not
    difficult to see that there is no monomial polyhedron  satisfying the hypotheses of Theorem~\ref{thm-irreg2} (i.e. $m=2$, the monomial $e_{(1,1)}(z)=z_1z_2$ is $\Gamma$-invariant and $\Uu_B$ is nontrivial). It is still unclear to us whether there exists $\Uu_B\subset\cx^2$ on which the Bergman projection fails to be of restricted weak type $(p_*,p_*)$ (or weak-type $(q_*,q_*)$). 
 
 \item It remains to be seen what conditions on the matrix $B$ of $\Uu_B$ characterize the precise endpoint $L^p$ behavior of $P_{\Uu_B}$ for $p=p_*$ and $p=q_*$.
 \item It would also be interesting to study the restricted type estimates for the Bergman projection on other quotient domains such as the symmetrized polydiscs.
 \end{enumerate}

\bibliographystyle{alpha}
\bibliography{references}

\newcommand{\etalchar}[1]{$^{#1}$}
\begin{thebibliography}{BCEM22}

\bibitem[Alm23]{MR4630461}
Rasha Almughrabi.
\newblock Bergman kernels of two dimensional monomial polyhedra.
\newblock {\em Complex Anal. Oper. Theory}, 17(6):Paper No. 103, 21, 2023.

\bibitem[Bar84]{barrett84}
David~E. Barrett.
\newblock Irregularity of the {B}ergman projection on a smooth bounded domain
  in {${\bf C}^{2}$}.
\newblock {\em Ann. of Math. (2)}, 119(2):431--436, 1984.

\bibitem[BCEM22]{CJM}
Chase Bender, Debraj Chakrabarti, Luke Edholm, and Meera Mainkar.
\newblock {$L^p$}-regularity of the {B}ergman projection on quotient domains.
\newblock {\em Canad. J. Math.}, 74(3):732--772, 2022.

\bibitem[BL80]{bellligocka}
Steve Bell and Ewa Ligocka.
\newblock A simplification and extension of {F}efferman's theorem on
  biholomorphic mappings.
\newblock {\em Invent. Math.}, 57(3):283--289, 1980.

\bibitem[CCG{\etalchar{+}}24]{summerpaper2023}
Debraj Chakrabarti, Isaac Cinzori, Ishani Gaidhane, Jonathan Gregory, and Mary
  Wright.
\newblock Bergman kernels of monomial polyhedra.
\newblock {\em J. Math. Anal. Appl.}, 531(1, part 2):Paper No. 127723, 28,
  2024.

\bibitem[CE24]{advances}
Debraj Chakrabarti and Luke~D. Edholm.
\newblock Projections onto {$L^p$}-{B}ergman spaces of {R}einhardt domains.
\newblock {\em Adv. Math.}, 451:Paper No. 109790, 46, 2024.

\bibitem[Che17]{chen17}
Liwei Chen.
\newblock The {$L^p$} boundedness of the {B}ergman projection for a class of
  bounded {H}artogs domains.
\newblock {\em J. Math. Anal. Appl.}, 448(1):598--610, 2017.

\bibitem[CK23]{weak3}
Adam~B. Christopherson and Kenneth~D. Koenig.
\newblock Weak-type regularity of the {B}ergman projection on rational
  {H}artogs triangles.
\newblock {\em Proc. Amer. Math. Soc.}, 151(4):1643--1653, 2023.

\bibitem[CK24]{weak4}
Adam~B. Christopherson and Kenneth~D. Koenig.
\newblock Endpoint {B}ehavior of the {B}ergman {P}rojection on a {C}lass of
  {$n$}-dimensional {H}artogs {T}riangles.
\newblock {\em Complex Anal. Oper. Theory}, 18(8):Paper No. 180, 2024.

\bibitem[CKY20]{CKY19}
Liwei Chen, Steven~G. Krantz, and Yuan Yuan.
\newblock {$L^p$} regularity of the {B}ergman projection on domains covered by
  the polydisc.
\newblock {\em J. Funct. Anal.}, 279(2):108522, 20, 2020.

\bibitem[CZ16]{chakzeytuncu}
Debraj Chakrabarti and Yunus~E. Zeytuncu.
\newblock {$L^p$} mapping properties of the {B}ergman projection on the
  {H}artogs triangle.
\newblock {\em Proc. Amer. Math. Soc.}, 144(4):1643--1653, 2016.

\bibitem[DHZZ01]{deng}
Yaohua Deng, Li~Huang, Tao Zhao, and Dechao Zheng.
\newblock Bergman projection and {B}ergman spaces.
\newblock {\em J. Operator Theory}, 46(1):3--24, 2001.

\bibitem[DM23]{dallara}
G.~Dall'Ara and A.~Monguzzi.
\newblock Nonabelian ramified coverings and {$L^p$}-boundedness of {B}ergman
  projections in {$\Bbb C^2$}.
\newblock {\em J. Geom. Anal.}, 33(2):Paper No. 52, 28, 2023.

\bibitem[EM16]{EdhMcN16}
L.~D. Edholm and J.~D. McNeal.
\newblock The {B}ergman projection on fat {H}artogs triangles: ${L}^p$
  boundedness.
\newblock {\em Proc. Amer. Math. Soc.}, 144(5):2185--2196, 2016.

\bibitem[EM17]{EdhMcN16b}
L.~D. Edholm and J.~D. McNeal.
\newblock Bergman subspaces and subkernels: degenerate ${L}^p$ mapping and
  zeroes.
\newblock {\em J. Geom. Anal.}, 27(4):2658--2683, 2017.

\bibitem[FK72]{follandkohn}
G.~B. Folland and J.~J. Kohn.
\newblock {\em The {N}eumann problem for the {C}auchy-{R}iemann complex}.
\newblock Princeton University Press, Princeton, N.J.; University of Tokyo
  Press, Tokyo, 1972.
\newblock Annals of Mathematics Studies, No. 75.

\bibitem[FR75]{rudin}
Frank Forelli and Walter Rudin.
\newblock Projections on spaces of holomorphic functions in balls.
\newblock {\em Indiana Univ. Math. J.}, 24:593--602, 1974/75.

\bibitem[GG23]{MR4627650}
Abhishek Ghosh and Gargi Ghosh.
\newblock {$L^p$} regularity of {S}zeg\"{o} projections on quotient domains.
\newblock {\em New York J. Math.}, 29:911--930, 2023.

\bibitem[Gra14]{grafakos1}
Loukas Grafakos.
\newblock {\em Classical {F}ourier analysis}, volume 249 of {\em Graduate Texts
  in Mathematics}.
\newblock Springer, New York, third edition, 2014.

\bibitem[Hun66]{hunt}
Richard~A. Hunt.
\newblock On {$L(p,\,q)$} spaces.
\newblock {\em Enseign. Math. (2)}, 12:249--276, 1966.

\bibitem[HW20]{huowickweaktype}
Zhenghui Huo and Brett~D. Wick.
\newblock Weak-type estimates for the {B}ergman projection on the polydisc and
  the {H}artogs triangle.
\newblock {\em Bull. Lond. Math. Soc.}, 52(5):891--906, 2020.

\bibitem[HW24]{huowicksymmetrized}
Zhenghui {Huo} and Brett~D. {Wick}.
\newblock {$L^p$ regularity of the Bergman projection on the symmetrized
  polydisc}.
\newblock {\em Canadian J. Math.}, October 2024.

\bibitem[JLQW24]{weak2}
Yuanyuan Jing, Yi~Li, Chuan Qin, and Mengjiao Wang.
\newblock Weak-type regularity of the {B}ergman projection on {$n$}-dimensional
  {H}artogs triangles.
\newblock {\em Complex Anal. Oper. Theory}, 18(1):Paper No. 15, 28, 2024.

\bibitem[LW24]{liwang}
Yi~Li and Mengjiao Wang.
\newblock Weak-type regularity for the {B}ergman projection over
  {$N$}-dimensional classical {H}artogs triangles.
\newblock {\em J. Inequal. Appl.}, pages Paper No. 39, 18, 2024.

\bibitem[McN94]{jeffbergmansingular}
Jeffery~D. McNeal.
\newblock The {B}ergman projection as a singular integral operator.
\newblock {\em J. Geom. Anal.}, 4(1):91--103, 1994.

\bibitem[MS94]{McNSte94}
J.~D. McNeal and E.~M. Stein.
\newblock Mapping properties of the {B}ergman projection on convex domains of
  finite type.
\newblock {\em Duke Math. J.}, 73(1):177--199, 1994.

\bibitem[NP09]{nagelpramanik}
Alexander Nagel and Malabika Pramanik.
\newblock Maximal averages over linear and monomial polyhedra.
\newblock {\em Duke Math. J.}, 149(2):209--277, 2009.

\bibitem[Par18]{park}
Jong-Do Park.
\newblock The explicit forms and zeros of the {B}ergman kernel for
  3-dimensional {H}artogs triangles.
\newblock {\em J. Math. Anal. Appl.}, 460(2):954--975, 2018.

\bibitem[PS77]{PhoSte77}
D.~H. Phong and E.~M. Stein.
\newblock Estimates for the {B}ergman and {S}zeg\"o projections on strongly
  pseudo-convex domains.
\newblock {\em Duke Math. J.}, 44(3):695--704, 1977.

\bibitem[QWZ25]{qin}
Chuan Qin, Maofa Wang, and Shuo Zhang.
\newblock {$L^p$}-norm {E}stimates of the {B}ergman {P}rojection on the
  {M}onomial {P}olyhedra.
\newblock {\em The Journal of Geometric Analysis}, 35(3):77, 2025.

\bibitem[SW59]{steinweiss1}
E.~M. Stein and Guido Weiss.
\newblock An extension of a theorem of {M}arcinkiewicz and some of its
  applications.
\newblock {\em J. Math. Mech.}, 8:263--284, 1959.

\bibitem[SW71]{steinweissbook}
Elias~M. Stein and Guido Weiss.
\newblock {\em Introduction to {F}ourier analysis on {E}uclidean spaces}.
\newblock Princeton Mathematical Series, No. 32. Princeton University Press,
  Princeton, NJ, 1971.

\bibitem[TTZ24]{weak1}
Yanyan Tang, Zhenhan Tu, and Shuo Zhang.
\newblock Weak-type estimates of the {B}ergman projection on a class of
  singular {R}einhardt domains.
\newblock {\em J. Math. Anal. Appl.}, 539(2):Paper No. 128626, 18, 2024.

\bibitem[Zha21a]{zhang1}
Shuo Zhang.
\newblock {$L^p$} boundedness for the {B}ergman projections over
  {$n$}-dimensional generalized {H}artogs triangles.
\newblock {\em Complex Var. Elliptic Equ.}, 66(9):1591--1608, 2021.

\bibitem[Zha21b]{zhang2}
Shuo Zhang.
\newblock Mapping properties of the {B}ergman projections on elementary
  {R}einhardt domains.
\newblock {\em Math. Slovaca}, 71(4):831--844, 2021.

\bibitem[ZJ64]{ZahJud64}
V.P. Zaharjuta and V.I. Judovic.
\newblock The general form of a linear functional on ${H}_p'$.
\newblock {\em Uspekhi Mat. Nauk.}, 19(2):139--142, 1964.

\end{thebibliography}

\end{document}